\documentclass[11pt,letter]{article}
\usepackage{amssymb,amsmath,amsthm,multirow,xcolor,cancel,bbm,tikz,mathrsfs,calc}
\usepackage[textheight=24cm,textwidth=16cm]{geometry}
\tikzstyle{arco}=[thick, -to]
\tikzstyle{main node}=[draw,circle,inner sep=1,outer sep=2,thick,minimum size=12pt]
\tikzstyle{pnode} =[fill,circle,minimum size =3pt, outer sep=2pt, inner sep = 0pt]
\newcommand{\bucle}[3][]{\draw[arco, #1](#2.#3-25)..controls +(#3-40:0.8) and +(#3+40:0.8).. (#2.#3+25);}

\newtheorem{definition}{Definition}
\newtheorem{proposition}{Proposition\setcounter{claimcounter}{0}}
\newtheorem{lemma}{Lemma\setcounter{claimcounter}{0}}
\newtheorem{theorem}{Theorem\setcounter{claimcounter}{0}}

\newcounter{claimcounter}
\newtheorem*{claim*}{{\it Claim}}

\newcommand{\fixe}{\text{\sc fix}}

\newcommand{\calP}{\mathcal{P}}

\newcommand{\calK}{K^*}

\newcommand{\B}{\{0,1\}}

\title{Number of fixed points and disjoint cycles\\ in monotone Boolean networks}

\author{
Julio Aracena\footnote{CI$^2$MA and Departamento de Ingenier\'ia Matem\'atica, Universidad de Concepci\'on, Av. Esteban Iturra s/n, Casilla 160-C, Concepci\'on, Chile. Email: \tt{jaracena@dim.uchile.cl}}
\footnote{Partially supported by FONDECYT project 1151265, Labex UCN@Sophia from the Universit\'e C\^ote d'Azur, France, BASAL project CMM, Universidad de Chile and by Centro de Investigaci\'on en Ingenier\'ia Matem\'atica (CI$^2$MA), Universidad de Concepci\'on.}
\and 
Adrien Richard\footnote{Laboratoire I3S, UMR CNRS 7271 \& Universit\'e de Nice-Sophia Antipolis, France. Email: \tt{richard@unice.fr}}~\footnote{Corresponding author.} 
\footnote{Partially supported by CNRS project PICS06718 and the project PACA
APEX FRI.}
\and 
Lilian Salinas\footnote{Department of Computer Sciences and CI$^2$MA, University of Concepci\'on, Edmundo Larenas 215, Piso 3, Concepci\'on, Chile \tt{lilisalinas@udec.cl}}
\footnote{Partially supported by FONDECYT project 1151265.}
}

\date{February 9, 2016; Revised April 5, 2017}

\begin{document}

\maketitle

\begin{abstract}
Given a digraph $G$, a lot of attention has been deserven on the maximum number $\phi(G)$ of fixed points in a Boolean network $f:\B^n\to\B^n$ with $G$ as interaction graph. In particular, a central problem in network coding consists in studying the optimality of the feedback bound $\phi(G)\leq 2^{\tau}$, where $\tau$ is the minimum size of a feedback vertex set of $G$. In this paper, we study the maximum number $\phi_m(G)$ of fixed points in a  {\em monotone} Boolean network with interaction graph $G$. We establish new upper and lower bounds on $\phi_m(G)$ that depends on the cycle structure of~$G$. In addition to $\tau$, the involved parameters are the maximum number $\nu$ of vertex-disjoint cycles, and the maximum number $\nu^*$ of vertex-disjoint cycles verifying some additional technical conditions. We improve the feedback bound $2^\tau$ by proving that $\phi_m(G)$ is at most the largest sub-lattice of $\B^\tau$ without chain of size $\nu+2$, and without another forbidden pattern described by two disjoint antichains of size $\nu^*+1$. Then, we prove two optimal lower bounds: $\phi_m(G)\geq \nu+1$ and $\phi_m(G)\geq 2^{\nu^*}$. As a consequence, we get the following characterization: $\phi_m(G)=2^\tau$ if and only if $\nu^*=\tau$. As another consequence, we get that if $c$ is the maximum length of a chordless cycle of $G$ then $2^{\nu/3^c}\leq\phi_m(G)\leq 2^{c\nu}$. Finally, with  the techniques introduced, we establish an upper bound on the number of fixed points of any Boolean network according to its signed interaction graph. 
\end{abstract}

\section{Introduction}\label{intro}

A {\em Boolean network} with $n$ components is a discrete dynamical system usually defined by a global transition function
\[
f:\B^n\to\B^n,\qquad x=(x_1,\dots,x_n)\mapsto f(x)=(f_1(x),\dots,f_n(x)).
\]
Boolean networks have many applications. In particular, since the seminal papers of McCulloch and Pitts \cite{MP43}, Hopfield \cite{H82}, Kauffman \cite{K69,K93} and Thomas \cite{T73,TA90}, they are omnipresent in the modeling of neural and gene networks (see \cite{B08,N15} for reviews). They are also essential tools in information theory, for the network coding problem \cite{ANLY00,GRF16}. 

\medskip
The structure of a Boolean network $f$ is usually represented via its {\em interaction graph}, which is the digraph $G$ with vertex set $\{1,\dots,n\}$ that contains an arc $uv$ if $f_v$ depends on $x_u$ ($G$ may have {\em loops}, that is, arcs from a vertex to itself).  

\medskip
In many contexts, as in molecular biology, the first reliable informations are represented under the form of an interaction graph, while the actual dynamics are very difficult to observe \cite{TK01,N15}. A natural question is then the following: {\em What can be said about  $f$ according to $G$ only?} Among the many dynamical properties that can be studied, fixed points are of special interest, since they correspond to stable states and often have a strong meaning. For instance, in the context of gene networks, they correspond to stable patterns of gene expression at the basis of particular biological processes \cite{TA90,A04}. As such, they are arguably the property which has been the most thoroughly studied. The study of the number of fixed points, and its maximization in particular, is the subject of a stream of work, e.g. in \cite{R86,RRT08,A08,GR11,ARS14,GRR15}. 

\medskip
Given a digraph $G$, let $\phi(G)$ be the maximum number of fixed points in a Boolean network whose interaction graph is (isomorphic to) $G$, and let $\phi'(G)$ be the maximum number of fixed points in a Boolean network whose interaction graph is (isomorphic to) a subgraph of $G$. A fundamental result is the following ``classical'' upper bound, called {\em feedback bound} in what follows, proved independently in different contexts in \cite{R07,A08}:
\[
\phi(G)\leq\phi'(G)\leq 2^\tau
\]
where $\tau=\tau(G)$ is the minimum size of a {\em feedback vertex set} of $G$, that is, the minimum size of a set of vertices intersecting every cycle of $G$ (cycles are always directed and without ``repeated vertices''). The optimality of this bound is a central problem in network coding: the so called {\em binary network coding problem} consists in deciding if $\phi'(G)=2^{\tau}$ \cite{R07,GR11}. Concerning lower bounds, let us mention the following ``folklore'' lower bound 
\[
2^\nu\leq \phi'(G)
\]
where $\nu=\nu(G)$ is the maximum number of vertex-disjoint cycles in $G$ (we have, obviously, $\nu\leq\tau$). This bound comes from the rather simple observation that if $G'$ is a subgraph of $G$ that consists in $k$ vertex-disjoint cycles, then $2^k=\phi(G')\leq \phi'(G)$. Up to our knowledge, lower bounds on $\phi(G)$ are known only for special class of  digraphs, as those with a loop on each vertex~\cite{GRF16}. See \cite{GR11} for other upper and lower bounds on $\phi'(G)$ based on other graph parameters. 

\medskip
In this paper we study fixed points in {\em monotone} Boolean networks, that is, functions $f:\B^n\to\B^n$ such that
\[
\forall x,y\in\B^n,\qquad x\leq y~\Rightarrow~ f(x)\leq f(y).
\]
Given a digraph $G$, we denote by $\phi_m(G)$ the maximum number of fixed points in a monotone Boolean network whose interaction graph is (isomorphic to) $G$, and we denote by $\phi'_m(G)$ the maximum number of fixed points in a monotone Boolean network whose interaction graph is (isomorphic to) a subgraph of $G$. Our main results, summarized below, are upper and lower bounds $\phi_m(G)$ and applications of the techniques introduced for these bounds to the more general context of signed digraphs. 

\medskip
We cannot speak about fixed points in monotone Boolean networks without first  mention the well-known Knaster-Tarski theorem, the following: {\em The set of fixed points of a monotone Boolean network $f$ is a non-empty lattice}. Using only the monotonicity of $f$, nothing else can be said on this non-empty lattice, since it is easy to see that {\em any} non-empty lattice contained in $\B^n$ may be the set of fixed points of $f$. However, one of the main contributions of this paper is to show that the interaction graph of $f$ contains useful additional information on the structure of the lattice of fixed points (and in particular on the maximal chains it contains). This is essentially from this additional information that our upper-bounds on $\phi_m(G)$ are derived.



\subsection{Results}

\subsubsection{Upper bounds}

Our first result is that, given a monotone Boolean network $f$ with interaction graph $G$, the set of  fixed points of $f$ is isomorphic to a lattice $L\subseteq \B^\tau$ without chain of size $\nu+2$. By a well known theorem of Erd\"os \cite{E45}, the maximal size of such a lattice is exactly two plus the sum of the $\nu-1$ largest binomial coefficients $\tau\choose k$. We thus obtain the following bound, which improve the feedback bound when the gap between $\nu$ and $\tau$ is large:
\[
\phi_m(G)\leq 2+ \sum_{k=\lfloor\frac{\tau-\nu+2}{2}\rfloor}^{\lfloor\frac{\tau+\nu-2}{2}\rfloor}{\tau \choose k}.
\]  
From this bound we recover the implication $\phi_m(G)=2^\tau~\Rightarrow~\nu=\tau$ already established in \cite{GRF16} with a dedicated proof. Since $2^\nu\leq \phi'_m(G)$ (this is an easy exercise), we deduce that 
\[
\phi'_m(G)=2^\tau \iff \nu=\tau.
\]
This solves the network coding problem in the binary monotone case and shows that, in this case, network coding do not outperform routing in terms of solvability; see \cite{GRF16} for details. 

\medskip
Then, we refine the upper bound by introducing another graph parameter, $\nu^*$, defined as follows. Let $C_1,\dots,C_k$ be a {\em packing} of size $k$, that is, a collection of $k$ vertex-disjoint cycles. A path $P$ is said {\em principal} if no arc and no internal vertex of $P$ belong  to the packing. We say that the packing is {\em special} if the following holds: for every cycle $C_i$ and every vertex $v$ in $C_i$, if there exists a principal path from a cycle $C_j\neq C_i$ to $v$, then there exists a principal path from $C_i$ or a source to $v$ (this path may start with $v$, and thus it may be a cycle). We then define $\nu^*=\nu^*(G)$ as the maximum size of a special packing of $G$ (we have, obviously, $\nu^*\leq \nu$). Also, we say that a subset $X\subseteq \B^n$ has a {\em special $k$-pattern} if there exists a subset $S\subseteq X$ of size $k$ such that $1-x\in X$ for all $x\in S$, and $x\leq 1-y$ for all distinct $x,y\in S$. Our refinement is based on the following fact: the set of fixed points of $f$ is isomorphic to a lattice $L\subseteq \B^\tau$ without special $(\nu^*+1)$-pattern. This has several consequences that we will discuss. In particular, since $\B^\tau$ has a $\tau$-pattern, we deduce that $\phi_m(G)=2^\tau~\Rightarrow ~\nu^*=\tau$. Since $\nu^*\leq\nu\leq\tau$, this improves the implication $\phi_m(G)=2^\tau~\Rightarrow~ \nu=\tau$ mentioned above. 

\subsubsection{Lower bounds}

Next, we prove several lower bounds by construction. The first is
\[
\nu+1\leq \phi_m(G),
\] 
which is optimal for every possible value of $\nu$. Since $2^\nu\leq \phi'_m(G)$, this optimality is not obvious a priori. Actually, we will prove that $\phi_m(G)=\nu+1$ if $G$ has an arc between any two vertex-disjoint cycles and if the maximum in-degree of $G$ is two.

\medskip
The second lower bound is
\[
2^{\nu^*}\leq \phi_m(G),
\] 
which is also optimal for every possible value of $\nu^*$ (consider for instance digraphs that consist in disjoint cycles). The proof consists in the construction of a Boolean network, in which, in some way, the principal paths of the special packing cancel each other, making the $\nu^*$ disjoint cycles independent each other. Then, they produce $2^{\nu^*}$ fixed points, as the union of $\nu^*$ disjoint cycles do (for comparison, the $2^\nu$ lower bound on $\phi'(G)$ is simply obtained by ignoring arcs that are not in the $\nu$ disjoint cycles).

\medskip
As a consequence, if $\nu^*=\tau$ then $\phi_m(G)=2^\tau$ and since, as seen above, the converse is true we have the following characterization, showing that the introduction of $\nu^*$ really makes sense:
\[
\phi_m(G)=2^\tau\iff \nu^*=\tau.
\] 
This characterization solves, in the binary monotone case, the {\em strict} network coding problem, introduced in \cite{GRF16} for studying arcs that are detrimental to network coding. In the binary case, the problem consists in deciding if $\phi(G)=2^{\tau}$. An interpretation resulting from the two above characterizations is then the following: in the monotone case, some arcs are detrimental to binary network coding if and only if $\nu^*<\nu=\tau$; see \cite{GRF16} for details. Note also that, in the context of gene networks, arcs come from experimental data and are thus usually not ignored: results concerning $\phi_m(G)$ are then more relevant than those concerning $\phi'_m(G)$.

\medskip
Finally we will prove, using both the lower bound $2^{\nu^*}$ and a result on dominating sets in digraphs, that if $G$ has $k$ vertex-disjoint cycles of length at most $\ell$ then $2^{k/3^\ell}\leq\phi_m(G)$. Hence, if $c=c(G)$ denotes the {\em circumference} of $G$, that is the maximum length of a chordless cycle of $G$, then $2^{\nu/3^c}\leq\phi_m(G)$, and since $\tau\leq c\nu$ this gives
\[
2^{\nu/3^c}\leq\phi_m(G)\leq 2^{c\nu}.
\]
Hence, if $G$ is symmetrical and loop-less, and thus identifiable with an undirected simple graph, then $2^{\nu/9}\leq\phi_m(G)$, and we will prove that the exponent $\nu/9$ can be replaced by $\nu/6$. More precisely, we will prove that $\nu/ 6\leq\nu^*$, and this immediately gives $2^{\nu/6}\leq\phi_m(G)$.

\subsubsection{Upper bound for signed digraphs}

Our motivation for studying fixed points in monotone Boolean networks comes from the studies about the relationships between the fixed points of Boolean network $f:\B^n\to\B^n$ and its {\em signed} interaction graph, which play a predominant role in biological and sociological applications \cite{TA90,TK01,H59}. 

\medskip
Formally, the {\em signed interaction graph} of $f$ is obtained by associating to its interaction graph $G$ the arc-labeling function $\sigma$ defined for all arc $uv$ as follows: 
\[
\sigma(uv)=
\left\{
\begin{array}{rl}
1  &\text{if $f_v(x)\leq f_v(x+e_u)$ for all $x\in\B^n$ with $x_u=0$},\\
-1 &\text{if $f_v(x)\geq f_v(x+e_u)$ for all $x\in\B^n$ with $x_u=0$},\\
0  &\text{otherwise.}
\end{array}
\right.
\]
(where $e_u$ denotes the point of $\B^n$ whose components are all zeros, except the one indexed by $u$). The {\em sign of a cycle} is the product of the signs of its arcs. Note that $f$ is monotone if and only if all the arcs are labeled positive. If $(G,\sigma)$ is a signed digraph then $\phi(G,\sigma)$ denotes the maximum number of fixed points in a Boolean network with a signed interaction graph isomorphic to $(G,\sigma)$. Let $\tau^+=\tau^+(G,\sigma)$ be the minimum size of a set of vertices intersecting every non-negative cycle of $(G,\sigma)$.

\medskip
The biologist Thomas put the emphasis on the dual role of positive and negative cycles, roughly:  positive cycles are the key ingredients for fixed point multiplicity, and negative cycles for some kind of sustained oscillations \cite{TA90,TK01}. One of the authors of this paper proved the following upper bound, that generalizes the classic one, and that gives a strong support to Thomas' ideas concerning positive cycles~\cite{A08}:
\[
\phi(G,\sigma)\leq 2^{\tau^+}.
\] 
The optimality of this bound is a difficult problem, and we think that it could be improved by taking, in some way, information on negative cycles. This was our initial motivation. For that problem, it is natural to study the extreme case where all cycles are positive, and this essentially corresponds to the monotone case. Indeed, as explained latter, if $(G,\sigma)$ is a strongly connected signed digraph with only positive cycles then $\phi(G,\sigma)=\phi_m(G)$ and all the previous results apply\footnote{If $(G,\sigma)$ has only positive cycles and $G$ is not strongly connected, then we may have $\phi(G,\sigma)>\phi_m(G)$. Thus the strong connectivity is necessary to ensure the equality.}.

\medskip
Actually, using tools introduced to bound $\phi_m(G)$ we will establish an upper bound that works for every signed digraph and which is competitive with $2^{\tau^+}$. Let us say that $(G,\sigma)$ and $(G,\sigma')$ are {\em equivalent} if there exists a set of vertices $I$ such that, for every arc $uv$: $\sigma'(uv)=\sigma(uv)$ if $u,v\in I$ or $u,v\not\in I$; and $\sigma'(uv)=-\sigma(uv)$ otherwise. A {\em monotone feedback vertex set} of $(G,\sigma)$ is then a feedback vertex set $I$ of $G$ such that every non-positive arc of $(G,\sigma)$ has its final vertex in $I$. Let $\tau^*_m=\tau^*_m(G,\sigma)$ the minimum size of a monotone feedback vertex set in a signed digraph equivalent to $(G,\sigma)$. We will prove that $\phi(G,\sigma)$ is at most the sum of the $\nu^++1$ largest binomial coefficients ${\tau^*_m\choose k}$, where $\nu^+$ is the maximum number of vertex-disjoint non-negative cycles in $(G,\sigma)$. In other words, for every signed digraph $(G,\sigma)$, 
\[
\phi(G,\sigma)\leq \sum_{k=\lfloor\frac{\tau^*_m-\nu^+}{2}\rfloor}^{\lfloor\frac{\tau^*_m+\nu^+}{2}\rfloor}{\tau^*_m \choose k}.
\]  

For instance, this improves the bound $2^{\tau^+}$ when $\tau^*_m=\tau^+$ and when the gap between $\nu^+$ and $\tau^*_m$ is large. This holds, for example, if $(G,\sigma)$ is the all-negative loop-less complete digraph $(K_n,-)$, that is, the loop-less signed digraph on $n$ vertices with $n^2-n$ negative arcs. In that case $\tau^*_m=\tau=\tau^+$ and $\nu^+=\lfloor\frac{\tau+1}{2}\rfloor$.

\subsection{Notations and basic definitions}

Unless otherwise specified, all the cycles and paths we consider are always directed and without repeated vertices, except that in a path $P=v_0v_1\dots v_\ell$ we may have $v_0=v_\ell$ (and, in that case, the path is actually a cycle). The vertices $v_1,\dots,v_{\ell-1}$, if they exist, are the {\em internal vertices} of~$P$. A {\em chord} in a cycle $C$ is an arc $uv$ such that $u$ an $v$ are in $C$ but not $uv$. A {\em loop} is an arc from a vertex to itself (a cycle of length one). Given a digraph $G=(V,A)$, the subgraph of $G$ induced by a set $I\subseteq V$ is denoted $G[I]$ and $G\setminus I=G[V\setminus I]$. 
A {\em strongly connected component} is a set of vertices $S\subseteq V$ such that $G[S]$ is strongly connected and such that $S$ is maximal for this property (with respect to the inclusion). The set of in-neighbors of a vertex $v$ is denoted $N^-_G(v)$ (and $v\in N^-_G(v)$ if there is a loop on $v$). The in-degree of $v$ is $|N^-_G(v)|$. A {\em source} is a vertex of in-degree zero.   

\medskip
For every positive integer $n$ we set $[n]=\{1,\dots,n\}$. We view $\B^n$ as the $n$-dimensional Boolean lattice, and every subset $X\subseteq \B^n$ is viewed as a poset with the usual partial order $\leq$ ({\em i.e.} $x\leq y$ if and only if $x_i\leq y_i$ for all $i\in [n]$). Two subsets $X,Y\subseteq \B^n$ are {\em isomorphic} if there exists a bijection from $X$ to $Y$ preserving the order. If $x\in\B^n$ then $\overline{x}=1-x$ is the negation of $x$ ({\em i.e.} $\overline{x}_i=1-x_i$ for all $i\in [n]$). If $I\subseteq [n]$ then $x_I$ is the restriction of $x$ to the components in $I$, so that $x_I\subseteq \B^I$. For every $I\subseteq [n]$ we denote by $e_I$ the point of $\B^n$ such that $(e_I)_v=1$ if and only if $v\in I$. We use $e_v$ as an abbreviation of $e_{\{v\}}$. For all $0\leq k\leq n$ we set denote by ${[n]\choose k}$ the set of $x\in\B^n$ containing exactly $k$ ones. 

\medskip
Let $f:\B^n\to\B^n$. For all $v\in [n]$, we say that the component $f_v:\B^n\to\B$ of $f$ is monotone if, for all $x,y\in\B^n$,  $x\leq y$ implies $f_v(x)\leq f_v(y)$. Thus $f$ is monotone if and only if all its components are. The main tool to study $f$ is its interaction graph $G$, whose definition is detailed here: the vertex set of $G$ is $[n]$, and for all $u,v\in[n]$, there is an arc $uv$ if $f_v$ depends on $x_u$, that is, if there exists $x,y\in\B^n$ that only differ in $x_u\neq y_u$ such that $f_v(x)\neq f_v(y)$. We denote by $\fixe(f)$ the set of fixed points of $f$. 

\subsection{Organization}

Upper and lower bounds on $\phi_m(G)$ are presented in Section~\ref{sec:upperbound} and \ref{sec:lowerbound} respectively. Results on signed digraphs are presented Section~\ref{sec:signedgraph}. Concluding remarks are given in Section~\ref{sec:conclusion}. 

\section{Upper bound for monotone networks}\label{sec:upperbound}

\subsection{Feedback bound and isomorphism}

The following lemma is a classical result proved in \cite{A08} from which we immediately deduce the feedback bound $\phi(G)\leq 2^\tau$ (it is sufficient to take $I$ such that $|I|=\tau$). Below, we will refine this lemma under some additional conditions about the monotony of the components of~$f$. Both proofs are almost identical, and to show that explicitly both proofs are presented. 

\begin{lemma}
Let $f$ be a Boolean network and let $I$ be a feedback vertex set of its interaction graph. Then
\[
\forall x,y\in\fixe(f),\qquad x= y\iff x_I= y_I.
\]
\end{lemma}

\begin{proof}
The direction $\Rightarrow$ is obvious. To prove the converse, suppose that $x_I= y_I$ and let us prove that $x= y$. Let $G$ be the interaction graph of $f$ and let $v_1,\dots,v_m$ be an enumeration of the vertices of $G\setminus I$ in the topological order (in such a way that there is no arc from $v_j$ to $v_i$ if $i<j$). We prove, by induction on $i$, that $x_{v_i}= y_{v_i}$. Since $f_{v_1}$ only depends on variables with indices in $I$ we can write $x_{v_1}=f_{v_1}(x)=f_{v_1}(x_I)$ and $y_{v_1}=f_{v_1}(y)=f_{v_1}(y_I)$. Since $x_I= y_I$ we deduce that $x_{v_1}=y_{v_1}$. Now, let $1<i\leq m$. Then, similarly, $f_{v_i}$ only depends on variables with indices in $J=I\cup\{v_1,\dots,v_{i-1}\}$ and thus we can write $x_{v_i}=f_{v_i}(x)=f_{v_i}(x_J)$ and $y_{v_i}=f_{v_i}(y)=f_{v_i}(y_J)$. By induction hypothesis $x_J= y_J$, thus $x_{v_i}=y_{v_i}$. This proves that $x=y$. 
\end{proof}

\begin{lemma}\label{lem:iso}
Let $f$ be a Boolean network and let $I$ be a feedback vertex set of its interaction graph. If $f_v$ is monotone for all $v\not\in I$ then 
\[
\forall x,y\in\fixe(f),\qquad x\leq y\iff x_I\leq y_I.
\]
As a consequence, $\fixe(f)$ is isomorphic to $\{x_I:x\in\fixe(f)\}$
\end{lemma}

\begin{proof}
The direction $\Rightarrow$ is obvious. To prove the converse, suppose that $x_I\leq y_I$ and let us prove that $x\leq y$. Let $G$ be the interaction graph of $f$ and let $v_1,\dots,v_m$ be an enumeration of the vertices of $G\setminus I$ in the topological order (in such a way that there is no arc from $v_j$ to $v_i$ if $i<j$). We prove, by induction on $i$, that $x_{v_i}\leq y_{v_i}$. Since $f_{v_1}$ only depends on variables with an index in $I$ we can write $x_{v_1}=f_{v_1}(x)=f_{v_1}(x_I)$ and $y_{v_1}=f_{v_1}(y)=f_{v_1}(y_I)$. Since $f_{v_1}$ is monotone and $x_I\leq y_I$ we deduce that $x_{v_1}\leq y_{v_1}$. Now, let $1<i\leq m$. Then, similarly, $f_{v_i}$ only depends on variables with an index in $J=I\cup\{v_1,\dots,v_{i-1}\}$ and we can write $x_{v_i}=f_{v_i}(x)=f_{v_i}(x_J)$ and $y_{v_i}=f_{v_i}(y)=f_{v_i}(y_J)$. Since $f_{v_i}$ is monotone and since, by induction hypothesis, $x_J\leq y_J$, we deduce that $x_{v_i}\leq y_{v_i}$. This proves that $x\leq y$. 
\end{proof}

\subsection{Introduction of {\boldmath$\nu$\unboldmath} in the upper bound}

The next lemma is a simple but very useful (and apparently new) application of a well known theorem, due to one of the authors \cite[Theorem 3]{A08} (see also Remy-Ruet-Thieffry \cite{RRT08}): {\em If a Boolean network $f$ has two distinct fixed points $x$ and $y$, then its signed interaction graph has a non-negative cycle $C$ such that $x_v\neq y_v$ for every vertex $v$ of $C$.} 

\begin{lemma}\label{lem:totalorder}
If the set of fixed points of a Boolean network $f$ contains a chain of size $\ell+1$, then the signed interaction graph of $f$ has $\ell$ vertex-disjoint non-negative cycles. 
\end{lemma}

\begin{proof}
Suppose indeed that $x^0< x^1< \cdots <x^\ell$ is a chain of fixed points of size $\ell+1$. For all $k\in [\ell]$, let $I_k$ be the set of components that differ between $x^{k-1}$ and $x^k$. Then, by the theorem mentioned above, for all $k\in [\ell]$, the signed interaction graph of $f$ has a non-negative cycle with only vertices in $I_k$. Since the sets $I_1,\dots,I_\ell$ are mutually disjoint, this proves the lemma.     
\end{proof}

The following theorem is a straightforward application of the previous lemmas and two well known theorems in combinatorics. The first is the Knaster-Tarski theorem stated in the introduction. The second is a theorem of Erd\H{o}s \cite{E45}: {\em For every $1\leq \ell\leq n$, the maximal size of a subset of $\B^n$ without chain of size $\ell+1$ is the sum of the $\ell$ largest binomial coefficients ${n\choose k}$.} 

\begin{theorem}\label{thm:erdos}
Let $f$ be a monotone Boolean network with interaction graph $G$. The number of fixed points in $f$ is at most $2$ plus the sum of the $\nu-1$ largest binomial coefficients $\tau\choose k$.
\end{theorem} 

\begin{proof}
By the Knaster-Tarski theorem, $\fixe(f)$ is a lattice. By Lemma~\ref{lem:iso}, this lattice is isomorphic to a lattice $L\subseteq \B^\tau$, and by Lemma~\ref{lem:totalorder}, $L$ has no chain of size $\nu+2$. Let $L'$ be obtained from $L$ by removing its maximal and minimal element. Since every maximal chain in $L$ contains these two extremal elements, $L'$ has no chain of size $\nu$. It is then sufficient to apply the Erd\H{o}s' theorem mentioned above. 
\end{proof}

An equivalent statement is the one given in the introduction:
\begin{equation}\label{eq:upperbound}
\phi_m(G)\leq 2+ \sum_{k=\lfloor\frac{\tau-\nu+2}{2}\rfloor}^{\lfloor\frac{\tau+\nu-2}{2}\rfloor}{\tau \choose k}.
\end{equation}
Hence, if $\phi_m(G)=2^\tau$ then $\phi_m(G)$ equals $2$ plus the sum of the $\tau-1$ largest binomial coefficients of order $\tau$, and thus $\nu-1\geq \tau-1$. Since $\nu\leq\tau$ this implies $\nu=\tau$. Thus we recover the implication $\phi_m(G)=2^\tau\Rightarrow\nu=\tau$ established in \cite{GRF16} with a dedicated proof. Another straightforward consequence is:
\begin{equation}\label{eq:nu1}
\nu\leq 1~\Rightarrow~ \phi_m(G)\leq 2.
\end{equation}

\medskip
The above upper bound on $\phi_m(G)$ improves the feedback Êbound $2^\tau$ when the gap between $\nu$ and $\tau$ is large. But for a fixed $\nu$, $\tau$ cannot be arbitrarily large. This results from the fact that directed cycles have the so called Erd\H{o}s-P\'osa property: as proved by Reed, Roberston, Seymour and Thomas \cite{RRST95}, there exists a smallest function $h:\mathbb{N}\to\mathbb{N}$ such that $\tau\leq h(\nu)$ for every digraph $G$. The only known exact value of $h$ is $h(1)=3$ \cite{M93}, and \cite{RRST95} provides an upper bound on $h(\nu)$ which is of power tower type. It is however difficult to find a family of digraphs for which the gap between $\nu$ and $\tau$ is large. To our knowledge, the best construction is by Seymour \cite{S95}, showing that for $n$ sufficiently large there exists a digraph on $n$ vertices with $\tau\geq\frac{1}{30}\nu\log\nu$. Let us also mention that $\nu=\tau$ for every strongly planar digraph~\cite{GT11}. These subtile relationships between $\nu$ and $\tau$ make particularly hard the study of the optimality of the upper bound on $\phi_m(G)$.  

\medskip
We will now identify another structure in $\fixe(f)$ that forces the presence of disjoint cycles. We need the following definition. 

\begin{definition}
A {\em $k$-pattern} in a subset $P\subseteq \B^n$ is a couple of two sequences $(x^1,\dots,x^k)$ and $(y^1,\dots,y^k)$, each of $k$ distinct elements of $P$, such that, for all $p,q\in [k]$,  
\[
x^p\leq y^q~\iff~p\neq q. 
\]
\end{definition}

For instance, the $n$ base vectors $(e_1,\dots,e_n)$ and their negations $(\overline{e_1},\dots,\overline{e_n})$ forms an $n$-pattern of $\B^n$. Also, if $x\in\B^n$ and $x\neq 0$, then $(x,\overline x)$ and $(\overline x, x)$ form a $2$-pattern of $\B^n$. This shows that the two sequences may have common elements. Additional properties on $k$-patterns are given below. 

\begin{proposition}\label{pro:pattern}
If $X=(x^1,\dots,x^k)$ and $Y=(y^1,\dots,y^k)$ form a $k$-pattern of $\B^n$, then $X$ and $Y$ are antichains and $k\leq n$.
\end{proposition}

\begin{proof}
There is nothing to prove if $k=1$ so assume that $k\geq 2$. If $x^p\leq x^q$ for some $p\neq q$ then $x^p\leq x^q\leq y^p$, which is a contradiction. Thus $X$ is an antichain, and we prove similarly that $Y$ is an antichain. 
Let us prove now that $k\leq n$. For all $p\in[k]$ there exists at least one $v\in[n]$ such that $x^p_v=1$ since otherwise $x^p\leq y^p$. Furthermore, if $x^p_v=1$ and $x^q_v=1$ for some $q\neq p$ then $y^p_v=1$ since $x^q\leq y^p$. Thus, if for all $v$ such that $x^p_v=1$ there exists some $q\neq p$ with $x^q_v=1$ then $x^p\leq y^p$, a contradiction. Thus, for every $p\in [k]$, there is at least one index in $[n]$, say $v_p$, such that $x^p_{v_p}>x^q_{v_p}$ for all $q\neq p$. Thus the $v_p$ are all distinct and we deduce that $k\leq n$.  
\end{proof}

\begin{lemma}\label{lem:pattern}
Let $f$ be a monotone network with interaction graph $G$. If the set of fixed points of $f$ has a $k$-pattern, then $G$ has $k$ vertex-disjoint cycles.
\end{lemma}

\begin{proof}
Suppose that $(x^1,\dots,x^k)$ and $(y^1,\dots,y^k)$ form a $k$-pattern in $\fixe(f)$. For every $p\in [k]$ let $V_p$ be the set of vertices $v$ in $G$ such that $x^p_v>y^p_v$. This set $V_p$ is not empty since $x^p\not\leq y^p$. Furthermore, if $v\in V_p$, then $f_v(x^p)=x^p_v>y^p_v=f_v(y^p)$ and since $f_v$ is monotone, we deduce that $G$ has an arc $uv$ with $x^p_u> y^p_u$. Thus $G[V_p]$ has minimum in-degree at least one, and thus $G[V_p]$ has a cycle, say $C_p$. Let us prove that the $k$ cycles $C_1,\dots,C_k$ selected in this way are vertex-disjoint. Let $p,q\in [k]$ with $p\neq q$. If $v\in V_p$ then $x^p_v>y^p_v$ and since $y^p\geq x^q$ we have $x^q_v=0$ and thus $v\not\in V_q$. Hence $V_p\cap V_q=\emptyset$, thus $C_p$ and $C_q$ are indeed vertex-disjoint. 
\end{proof}

Suppose that $f$ has $2^\tau$ fixed points. Then $\fixe(f)$ is isomorphic to $\B^\tau$, which has a $\tau$-pattern. Thus, according to the previous lemma, we have $\nu=\tau$ and we recover, by another way, the implication $\phi_m(G)=2^\tau\Rightarrow \nu=\tau$. Furthermore, we deduce from the previous lemma that $\B^\tau$ has no $(\tau+1)$-pattern (since otherwise $\nu>\tau$ which is false). This has been proved independently in Proposition~\ref{pro:pattern}.

\subsection{Introduction of {\boldmath$\nu^*$\unboldmath} in the upper bound}

We will now identify a slightly different structures that force the presence of disjoint cycles verifying some additional conditions. These structures are the {\em special $k$-patterns}, defined in a compact way in the introduction. Here is an equivalent and more explicit definition, based on $k$-patterns. 

\begin{definition}
A $k$-pattern $(x^1,\dots,x^k),(y^1,\dots,y^k)$ is {\em special} if, for all $p\in [k]$,  
\[
y^p=\overline{x^p}.  
\]
\end{definition}

For instance, the $n$ base vectors $(e_1,\dots,e_n)$ and their negations $(\overline{e_1},\dots,\overline{e_n})$ form a special $n$-pattern. This example can be generalized as follows. For every $1\leq\ell\leq \frac{n}{2}$ there always exists a collection $I_1,\dots,I_k$ of $k=\lfloor\frac{n}{l}\rfloor$ disjoint subsets of $[n]$ of size $\ell$. The sequences $(e_{I_1},\dots,e_{I_k})$ and $(\overline{e_{I_1}},\dots,\overline{e_{I_k}})$ then form a special $k$-pattern. As a consequence, for every $P\subseteq \B^n$, 
\begin{equation}\label{eq:specialpattern1}
\begin{array}{c}
\text{${[n]\choose \ell}\subseteq P$ and ${[n]\choose n-\ell}\subseteq P$}
~\Rightarrow~ 
\text{$P$ has a special $\lfloor\frac{n}{l}\rfloor$-pattern.}
\end{array}
\end{equation}
Conversely, every special $k$-pattern can obviously be expressed as $(e_{I_1},\dots,e_{I_k}),$ $(\overline{e_{I_1}},\dots,\overline{e_{I_k}})$ for some subsets $I_1,\dots,I_k$ of $[n]$. Then, necessarily, these sets $I_p$ are pairwise disjoint, since for all distinct $p,q\in[n]$, we have $e_{I_p}\leq\overline{e_{I_q}}$ and this is equivalent to $I_p\subseteq [n]\setminus I_q$. Note also that if there exists $x\in P$ such that $\overline{x}\in P$ then $P$ has a special $2$-pattern, formed by $(x,\overline{x})$ and $(\overline{x},x)$. Conversely, if $P$ has a spacial $2$-pattern then, by definition, there exists $x\in P$ such that $\overline{x}\in P$. Thus, setting $\overline{P}=\{\overline{x}:x\in P\}$, we have 
\begin{equation}\label{eq:specialpattern2}
\text{$P$ has no special $2$-pattern}~\iff~
P\cap\overline{P}=\emptyset.
\end{equation}


\medskip
We also recall, from the introduction, the definition of a {\em principal path} and a {\em special packing} (see Figures \ref{fig:ex1} and \ref{fig:nustar} for illustrations). 

\begin{definition}
A {\em packing} of size $k$ in a digraph $G$ is a collection of $k$ vertex-disjoint cycles of $G$, say $C_1,\dots,C_k$. Given such a packing, a path $P$ of $G$ is said {\em principal} if it has no arc and no internal vertex that belong to some cycle of the packing \footnote{If $P$ has some internal vertices, and if none of them are in the packing then, obviously, no arc of $P$ is in the packing. Thus the condition ``$P$ has no arc in the packing'' makes sense only when $P$ has no internal vertex, that is, when $P$ is reduced to a single arc.}. We say that the packing is {\em special} if the following holds: for every cycle $C_i$ and for every vertex $v$ in $C_i$, if there exists a principal path from a cycle $C_j\neq C_i$ to $v$, then there exists a principal path from $C_i$ or a source to $v$ (this path can eventually start with $v$, and thus it may be a cycle). We denote by $\nu^*=\nu^*(G)$ the maximum size of a special packing in $G$. 
\end{definition}

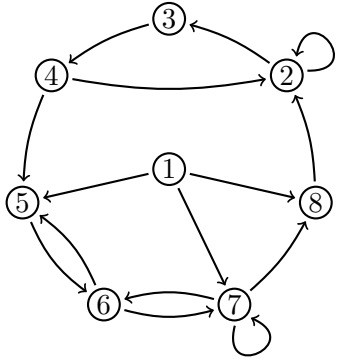
\begin{figure}
\centering
\begin{tikzpicture}
\node[main node] (1) at (0,0) {1};
\foreach \x in {2,...,8}{
	\node[main node] (\x) at (360*\x/7+90-3*360/7:2) {\x};
}
\bucle{7}{90-3*360/7}
\bucle{2}{360*2/7+90-3*360/7}
\path[arco]
(1) edge (5)
(1) edge (7)
(1) edge (8);
\path[arco, bend right =10]
(2) edge (3)
(4) edge (5)
(3) edge (4)
(4) edge (2)
(5) edge[bend right = 15] (6)
(6) edge[bend right = 15] (5)
(6) edge[bend right = 15] (7)
(7) edge[bend right = 15] (6)
(7) edge (8)
(8) edge (2)
;
\end{tikzpicture}
\caption{\label{fig:ex1}The cycles $C_1=2,3,4,2$, $C_2=5,6,5$ and $C_3=7,7$ form a packing which is not special: $P=7,6$ is a principal path from $C_3$ to $6$, but there is no principal path from $C_2$ to $6$ or from a source to $6$. However, the packing formed by $C_1$ and $C_2$ is a special packing. Indeed, $P=6,7,8,2$ is the unique principal path from $C_2$ to $C_1$, and $P'=2,2$ is a principal path from $C_1$ to $2$. Furthermore, $P=4,5$ is the unique principal path from $C_1$ to $C_2$ and there exists a principal path $P'=1,5$ from a source to $5$.}
\end{figure}

\begin{figure}
\def\n{6}
\def\r{1.5}
\[
\begin{array}{ccc}
\begin{array}{c}
\begin{tikzpicture}
\foreach \x in {1,...,\n}{
		\node[pnode] (\x) at ({\x*(360/\n)}:\r) {};
}
\foreach \x [evaluate=\x as \nextx using {int(mod(\x,\n)+1)}] in {1,...,\n}{
	\path[arco,bend right=12](\x) edge (\nextx);
	\bucle{\x}{\x*360/\n}
}
\end{tikzpicture}
\\
\nu^*=\lfloor n/2\rfloor \quad \nu=\tau=n
\end{array}
&\quad&
\begin{array}{c}
\begin{tikzpicture}
\node[pnode] (center) at (0,0) {};
\foreach \x in {1,...,\n}{
		\node[pnode] (\x) at ({\x*(360/\n)}:\r) {};
}
\foreach \x [evaluate=\x as \nextx using {int(mod(\x,\n)+1)}] in {1,...,\n}{
	\path[arco,bend right=12](\x) edge (\nextx);
	\bucle{\x}{\x*360/\n}
	\path[arco](center) edge (\x);
}
\end{tikzpicture}
\\
\nu^*=\nu=\tau=n-1.
\end{array}
\\[30mm]
\begin{array}{c}
\begin{tikzpicture}
\foreach \x in {1,...,\n}{
		\node[pnode] (\x) at ({\x*(360/\n)}:\r) {};
		\node[pnode] (-\x) at ({\x*(360/\n)}:{\r/2.5}) {};
}
\foreach \x [evaluate=\x as \nextx using {int(mod(\x,\n)+1)}] in {1,...,\n}{
	\bucle{\x}{\x*360/\n}
	\path[arco,bend right=12](\x) edge (\nextx);
	\path[arco,bend right=16](-\x) edge (\x)
	(\x) edge (-\x);
}
\end{tikzpicture}
\\
\nu^*=\nu=\tau=n/2
\end{array}
&\quad&
\begin{array}{c}
\begin{tikzpicture}
\node[pnode] (center) at (0,0) {};
\foreach \x in {1,...,\n}{
		\node[pnode] (\x) at ({\x*(360/\n)}:\r) {};
}
\foreach \x [evaluate=\x as \nextx using {int(mod(\x,\n)+1)}] in {1,...,\n}{
	\path[arco,bend right=13](center) edge (\x) (\x) edge (center);
}
\end{tikzpicture}
\\
\nu^*=\nu=\tau=1
\end{array}
\end{array}
\]
\caption{\label{fig:nustar} Parameters $\nu^*$, $\nu$ and $\tau$ for some digraphs.}
\end{figure}
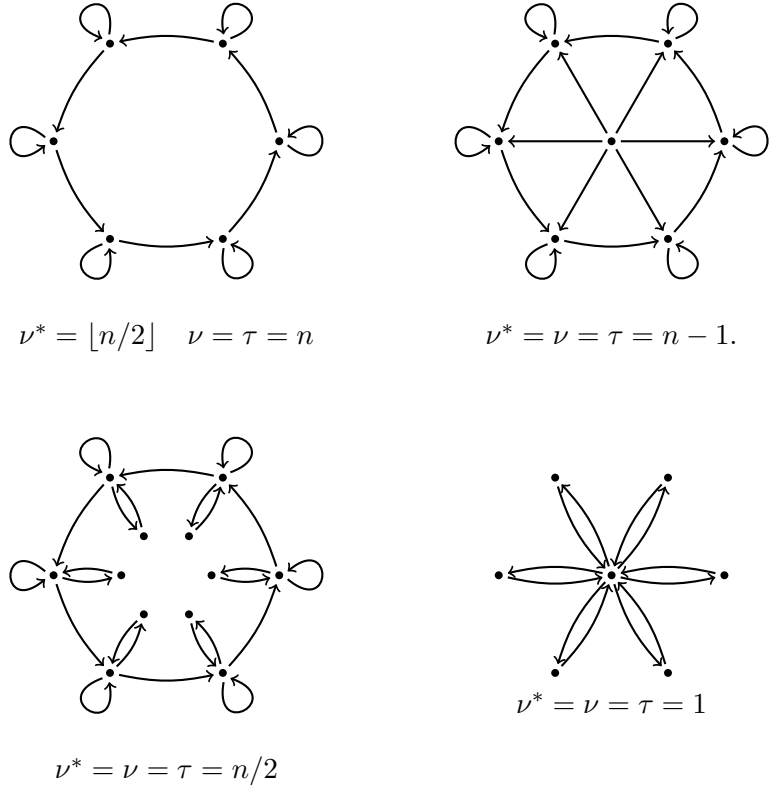

We now connect special patterns with special packings in the following way.  

\begin{lemma}\label{lem:two_strongly_independent_cycles}
Let $f$ be a monotone network with interaction graph $G$, and let $I$ be a feedback vertex set of $G$. If $\{x_I:x\in\fixe(f)\}$ has a special $k$-pattern, then $G$ has a special packing of size $k$. 
\end{lemma}

\begin{proof}
Let $X=(x^1,\dots,x^k)$ and $Y=(y^1,\dots,y^k)$ be two sequences of $k$ distinct fixed points of $f$, and suppose that $X_I=(x^1_I,\dots,x^k_I)$ and $Y_I=(y^1_I,\dots,y^k_I)$ form a special $k$-pattern of $L=\{x_I:x\in\fixe(f)\}$. Since, by Lemma~\ref{lem:iso}, $\fixe(f)$ and $L$ are isomorphic, $X$ and $Y$ form a $k$-pattern of $\fixe(f)$ (which is, however, not necessarily special). 

\medskip
For every $p\in [k]$ we set 
\[
V_p=\{v:x^p_v>y^p_v\}
\qquad\text{and}\qquad
U_p= \{v:x^p_v\geq y^p_v\}.
\]
We first prove that, for all distinct $p,q\in [k]$, 
\begin{equation}\label{eq:VU}
U_p\cap V_q=\emptyset.
\end{equation}
Indeed, if $p\neq q$ then, since $X$ and $Y$ from a $k$-pattern of $\fixe(f)$, we have $x^p\leq y^q$ and $x^q\leq y^p$. Hence, if $v\in V_q$ then $x^q_v>y^q_v\geq x^p_v=0$, so if $v\in U_p$ then $y^p_v=0$ which implies $x^q_v=0$, a contradiction. This proves (\ref{eq:VU}). 
  
\medskip
As already said in Lemma~\ref{lem:pattern}, if $v\in V_p$, then $f_v(x^p)=x^p_v>y^p_v=f_v(y^p)$ and since $f_v$ is monotone we deduce that $G$ has an arc $uv$ with $x^p_u>y^p_u$. Thus 
\begin{equation}\label{eq:firstpro}
\text{\em $G[V_p]$ has minimal in-degree at least one.}
\end{equation}
Furthermore, if $v\in U_p\setminus V_p$, that is $x^p_v=y^p_v=\alpha$ for some $\alpha\in\B$, then $f_v(x^p)=f_v(y^p)=\alpha$. Hence, if $f_v$ is not a constant then $G$ has an arc $uv$ with $u\in U_p$. Indeed, if $f_v$ is not a constant function and $x^p_u<y^p_u$ for every arc $uv$ of $G$, then $f_v(x^p)<f_v(y^p)$, a contradiction. We have thus the following property:
\begin{equation}\label{eq:secondpro}
\text{\em If $v\in U_p\setminus V_p$ and $f_v$ is not a constant function, then $G$ has an arc $uv$ with $u\in U_p$.}
\end{equation}

\medskip
Let $S^1,\dots,S^m$ be an enumeration of the strongly connected components of $G[U_p]$ in the topological order (that is, in such a way that $G[U_p]$ has no arc from $S^i$ to $S^j$ if $j<i$). Let $S^i$ be the first component containing a cycle. This component exists since, by property (\ref{eq:firstpro}), $G[V_p]$ has a cycle and $V_p\subseteq U_p$. Suppose that $G[U_p]$ has an arc $uv$ with $u\notin S^i$ and $v\in S^i$. Then $G[U_p]$ has a path $P$ starting from an initial component $S^j$ with $j<i$ and ending in $u$ (initial means there is no arc from $S^\ell$ to $S^j$ for every $1\leq \ell<j$). Then $S^j$ has no cycle (since $j<i$) thus $S^j=\{w\}$ for some vertex $w$ with in-degree zero in $G[U_p]$. We then deduce from (\ref{eq:secondpro}) that $f_w$ is a constant function, that is, $w$ is a source of $G$. Thus $P$ is a path from a source of $G$ to $u$ with only vertices in $U_p\setminus S^i$, and by adding the arc $uv$ we obtain a path from a source to $v$ with only internal vertices in $U_p\setminus S^i$. Setting $S_p=S^i$ we have thus the following property:
\begin{equation}\label{eq:thirdpro}
\begin{array}{p{11.5cm}}
\em If $G$ has an arc $uv$ with $v\in S_p$ and $u\in U_p\setminus S_p$ then $G$ has a path $P$ from a source of $G$ to $v$ with only internal vertices in $U_p\setminus S_p$.
\end{array}
\end{equation}
Since $X_I$ and $Y_I$ form a special $k$-pattern of $L$, we have $x^p_I=\overline{y^p_I}$ so that $U_p\cap I=V_p\cap I$. Hence $V_p\cap I$ is a feedback vertex set of $G[U_p]$ thus $S_p\cap V_p$ is not empty. According to (\ref{eq:firstpro}) if $v\in S_p\cap V_p$ then either $G[V_p]$ has a cycle $C$ containing $v$ or $G[V_p]$ has a cycle $C$ and a path from $P$ from $C$ to $v$. In both cases, by the choice of $S_p$, $C$ is necessarily contained in $G[S_p]$. Thus 
\[
\text{\em $G[S_p\cap V_p]$ contains a cycle.}
\]

\medskip
For every $p\in[k]$, let $C_p$ be a cycle of $G[S_p\cap V_p]$. From (\ref{eq:VU}) these $k$ cycles are mutually vertex-disjoint, so they form a packing, and it remains to prove that this packing is special. Suppose that $G$ has a principal path $P$ from $C_p$ to $C_q$ with $p\neq q$. Let $v$ be the last vertex of $P$, and let $u$ be the vertex preceding $v$ in $C_q$. To complete the proof, we have to prove that $G$ has a principal path from $C_q$ to $v$ or from a source of $G$ to $v$. We consider two cases. 
\begin{enumerate}
\item
Suppose first that $G$ has an arc $wv$ such that $w\in U_q$ and $w\neq u$. If $w$ is in $C_q$ then the chord $wv$ constitutes a principal path from $C_q$ to $v$. If $w$ is not in $C_q$ and $w\in S_q$, then, since $G[S_q]$ is strongly connected, $G[S_q]$ has a path from $C_q$ to $w$ which, together with the arc $wv$, gives a principal path from $C_q$ to $v$. If $w\not\in S_q$ then, according to (\ref{eq:thirdpro}), $G$ has a principal path from a source of $G$ to $v$. This completes the first case. 
\item
Suppose now that $v$ has no in-neighbor in $U_q$ except $u$. Let $w_1,\dots,w_r$ be an enumeration of the in-neighbors of $v$ distinct from $u$, and let us write $f_v(z)=f_v(z_u,z_{w_1},\dots,z_{w_r})$. We have $x^q_v>y^q_v$ and $x^q_u>y^q_u$ because $v,u\in V_q$. Also, for all $s\in[r]$, we have $x^q_{w_s}<y^q_{w_s}$ because $w_s\not\in U_q$. Since $f$ is monotone, we deduce that, for all $z_{w_1},\dots,z_{w_r}\in\B$, 
\[
\begin{array}{l}
f_v(1,z_{w_1},\dots,z_{w_r})\geq f_v(1,0,\dots,0)=f_v(x^q_u,x^q_{w_1},\dots,x^q_{w_k})=f_v(x^q)=x^q_v=1\\
f_v(0,z_{w_1},\dots,z_{w_r})\leq f_v(0,1,\dots,1)=f_v(y^q_u,y^q_{w_1},\dots,y^q_{w_k})=f_v(y^q)=y^q_v=0\\
\end{array}
\]  
Thus $f_v(z_u,z_{w_1},\dots,z_{w_r})=z_u$ for all $z_u,z_{w_1},\dots,z_{w_r}\in\B$. It means that  $f_v(z)$ only depends on $z_u$, so that $u$ is the unique in-neighbor of $v$ in $G$. Since $u$ is in $C_q$, this contradict the existence of the principal path $P$ from $C_p$ to $v$. Thus the second case is actually not possible, and this completes the proof. 
\end{enumerate}
\end{proof}

Summarizing the previous results we get the following. 

\begin{theorem}[Upper bound for monotone Boolean networks]\label{thm:upperbound}
If $f$ is a monotone Boolean network with interaction graph $G$, then the set of fixed points of $f$ is isomorphic to a subset $L\subseteq \B^\tau$ verifying the following conditions:
\begin{enumerate}
\item
$L$ is a lattice,
\item
$L$ has no chain of size $\nu+2$,
\item
$L$ has no $(\nu+1)$-pattern,
\item
$L$ has no special $(\nu^*+1)$-pattern. 
\end{enumerate}
\end{theorem}

Let us discuss some consequences of the fourth constraint, on special patterns and special packings. First, if $L=\{0,1\}^\tau$ then $L$ has a special $\tau$-pattern and thus $\nu^*\geq \tau$, and since $\nu^*\leq\tau$ we get $\nu^*=\tau$. This shows that, for every digraph $G$,
\begin{equation}\label{eq:oneway}
\phi_m(G)=2^\tau~\Rightarrow~\nu^*=\tau. 
\end{equation}
In the next section, we will prove that the converse is true, thus showing that the notions of special pattern and special packing are very natural in our context. 

\medskip
We can actually prove a more general implication. Let $x^-$ and $x^+$ be the minimal and maximal element of $L$, and let $L'=L\setminus \{x^-,x^+\}$. Since $L'$ has no chain of size $\nu$, as mentioned above, it results from a theorem of Erd\H{o}s~\cite{E45} that $|L'|$ is at most the sum of the $\nu-1$ largest binomial coefficients ${\tau\choose k}$. Suppose that $|L'|$ reaches this sum, and let $\ell=\frac{\tau-\nu+2}{2}$. Erd\H{o}s, F{\"u}redi and Katona \cite{EFK05} proved that $L'$ is then the union of the sets ${[\tau]\choose k}$ where $k$ ranges either in the interval $[\lfloor \ell\rfloor, \lfloor \tau-\ell\rfloor]$ or in the interval $[\lceil \ell\rceil,\lceil\tau-\ell\rceil]$. In both cases, ${[\tau]\choose \lceil \ell\rceil}$ and ${[\tau]\choose \tau-\lceil \ell\rceil}$ are contained in $L'$ so, by  (\ref{eq:specialpattern1}), $L'$ has a special $\lfloor\frac{\tau}{\lceil \ell\rceil}\rfloor$-pattern. As a consequence of the fourth constraint we get $\nu^*\geq \lfloor\frac{\tau}{\lceil \ell\rceil}\rfloor$. This shows that, for all digraphs $G$, 
\[
\phi_m(G)= 2+ \sum_{k=\lfloor\frac{\tau-\nu+2}{2}\rfloor}^{\lfloor\frac{\tau+\nu-2}{2}\rfloor}{\tau \choose k}
\quad\Rightarrow\quad
\nu^*\geq\left\lfloor\frac{\tau}{\left\lceil\frac{\tau-\nu+2}{2}\right\rceil}\right\rfloor
\]
In particular, if $\phi_m(G)=2^\tau$ then $\nu=\tau$ and we deduce that $\nu^*=\tau$. We thus recover (\ref{eq:oneway}). 

\medskip
It is important to note that, contrary to $\nu$, for a fixed $\nu^*$ the transversal number $\tau$ can be arbitrarily large. For instance, if $T'_n$ denotes the digraph obtained from the transitive tournament on $n$ vertices by adding a loop on each vertex, then we have $\nu^*(T'_n)=1$ and $\nu(T'_n)=\tau(T'_n)=n$. In the next section we will exhibit a family of strongly connected digraphs with a similar property (cf. Proposition~\ref{pro:Kn}). Thus bounding $\nu^*$ is much more weaker than bounding $\nu$, and considering the case $\nu^*=1$ could be interesting. 

\medskip
So suppose that $\nu^*=1$. According to the fourth constraint, $L$ has no special $2$-pattern, and thus $L'$ has no special $2$-pattern. Hence, by (\ref{eq:specialpattern2}), this is equivalent to say that $L'\cap\overline{L'}=\emptyset$. Let $t$ be the number of $v\in[\tau]$ such that $x^-_v<x^+_v$, and let $X$ be the set of $x\in\B^\tau$ such that $x^-< x< x^+$. Then $L'\subseteq X$ and $X$ is isomorphic to $\B^t\setminus\{0,1\}$. Since $L'\cap\overline{L'}=\emptyset$ we deduce that $|L'|\leq \frac{2^t-2}{2}\leq 2^{\tau-1}-1$ so that $|L|\leq 2^{\tau-1}+1$. Thus, for every digraph $G$,  
\begin{equation}\label{eq:nu^*=1}
\nu^*=1~\Rightarrow~ \phi_m(G)\leq 2^{\tau-1}+1.
\end{equation}
This bound is tight since, for every $n\geq 1$, we have $\nu^*(T'_n)=1$, $\tau(T'_n)=n$ and $\phi_m(T'_n)=2^{n-1}+1$. Indeed, let $f:\B^n\to\B^n$ be defined by $f_n(x)=x_n$ and $f_v(x)=x_v\lor \bigwedge_{v<u\leq n}x_u$ for all $1\leq v<n$. In this way, $f$ is monotone and has $T'_n$ as interaction graph. Furthermore, if $x_n=0$ then $f(x)=x$, and since $f(1)=1$, we deduce that $f$ has $2^{n-1}+1$ fixed points. Thus $\phi_m(T'_n)\geq 2^{n-1}+1$ and from \eqref{eq:nu^*=1} we obtain $\phi_m(T'_n)=2^{n-1}+1$. For comparison, it follows from \cite[Theorem 3]{GRF16} that, in the general case, $\phi(T'_n)=2^{n-1}+2^{n-2}$.


\section{Lower bound for monotone networks}\label{sec:lowerbound}

\subsection{Optimal linear lower bound in {\boldmath$\nu$\unboldmath}}

\begin{lemma}\label{lem:lowerbound}
For every digraph $G$,
\[
\nu+1\leq\phi_m(G).
\]
\end{lemma}

\begin{proof}
Suppose that the vertex set of $G$ is $[n]$. Let $C_1,\dots,C_\nu$ be vertex-disjoint cycles in $G$, and for every $p\in[\nu]$, let $V_p$ be the vertices of $C_p$. Let $U_0$ be the set of vertices that cannot be reached from one of these cycles ($U_0$ may be empty). Let $U_1$ be the set of vertices reachable from $C_1$ in $G\setminus (V_2\cup \dots\cup V_\nu)$. For $2\leq p\leq \nu$, let $U_p$ be the set of vertices reachable from $C_p$ in $G\setminus(U_1\cup\dots\cup U_{p-1}\cup V_{p+1}\cup \dots\cup V_\nu)$. In this way $\{U_0,U_1,\dots,U_\nu\}$ is a partition of the vertex set of $G$. Note that every $G[U_p]$ is of minimum in-degree at least one. 

\medskip
For every $0\leq p\leq\nu$ and every $v\in U_p$ we set 
\[
\theta_v=|N^-_G(v)\cap (U_0\cup\dots\cup U_p)|.
\]
Let $f:\B^n\to\B^n$ be defined for every $v\in[n]$ by: 
\[
f_v(x)=
\begin{cases}
1&\text{if }\sum_{u\in N^-_G(v)} x_u\geq\theta_v.\\
0&\text{otherwise.}
\end{cases}
\]
It is easy to check that $f$ is monotone and that $G$ is its interaction graph. For $0\leq q\leq \nu$, let $x^q\in\B^n$ be defined for every $v\in[n]$ by:  
\[
x^q_v=
\begin{cases}
1&\text{if }v\in U_0\cup \cdots\cup U_q\\
0&\text{otherwise.}
\end{cases}
\]
We claim that each $x^q$ is a fixed point of $f$. So let $0\leq q\leq\nu$, and let $v\in U_p$ with $0\leq p\leq\nu$. Then 
\[
\begin{array}{l}
\sum_{u\in N^-_G(v)} x^q_u=|N^-_G(v)\cap(U_0\cup\dots\cup U_q)|.
\end{array}
\]
Hence if $0\leq p\leq q$ this sum is at least $\theta_v$, and thus $f_v(x^q)=1=x^q_v$. If $q<p$ then this sum is at most $\theta_v-1$, because $N^-_G(v)\cap U_p\neq\emptyset$ (by construction $G[U_p]$ is of minimum in-degree at least one), and we deduce that $f_v(x^q)=0=x^q_v$. Thus $x^0,\dots,x^\nu$ are distinct fixed points of $f$.
\end{proof}

We will now prove that this lower bound is tight. For that we first identify a rather general class of networks without $\nu+2$ fixed points. Let $f:\B^n\to\B^n$ be a Boolean network with interaction graph $G$. We say that $f$ is an {\em and-or-network} if for every $v\in [n]$ we have either $f_v(x)=\wedge_{u\in N^-_G(v)}x_u$ for all $x\in\B^n$ or $f_v(x)=\vee_{u\in N^-_G(v)}x_u$ for all $x\in\B^n$ (by convention, the empty conjunction is $1$ and the empty disjunction is $0$). In the first case we say that $v$ is an {\em and-vertex} and in the second case we say that $v$ is an {\em or-vertex}. We say that two vertex-disjoint cycle $C_1$ and $C_2$ of $G$ are {\em independent} if $G$ has no arc from $C_1$ to $C_2$ and no arc from $C_2$ to~$C_1$. 

\begin{lemma}\label{lem:andortotalorder}
Let $f$ be an and-or-network with interaction graph $G$. If $G$ has no two independent cycles, then $f$ has at most $\nu+1$ fixed points. 
\end{lemma}

\begin{proof}
Suppose that $x^1$ and $x^2$ are two fixed points of $f$. Let $I^1$ be the set of vertices $v$ such that $x^1_v>x^2_v$, and let $I^2$ be the set of vertices $v$ such that $x^2_v>x^1_v$. Suppose that $x^1$ and $x^2$ are incomparable. Then both $I^1$ and $I^2$ are not empty and it is easy to check that both $G[I_1]$ and $G[I_2]$ are of minimum in-degree at least one. Thus $G[I_1]$ contains a cycle $C_1$ and $G[I_2]$ contains a cycle $C_2$. Furthermore, there is no arc $uv$ with $u\in I_1$ and $v\in I_2$. Suppose indeed that such an arc $uv$ exists. Since $x^1_u=1$ and $x^1_v=0=f_v(x^1)$, $v$ cannot be an or-vertex, and since $x^2_u=0$ and $x^2_v=1=f_v(x^2)$, $v$ cannot be an and-vertex, a contradiction. Thus there is no arc from $I_1$ to $I_2$ and by symmetry there is no arc from $I_2$ to $I_1$. Thus $C_1$ and $C_2$ are two independent cycles of $G$. We have thus proved the following: {\em If $f$ has two incomparable fixed points then $G$ has two independent cycles.} Therefore, if $G$ has no two independent cycles then $\fixe(f)$ is a chain, which is, by Lemma \ref{lem:totalorder}, of size at most $\nu+1$.
\end{proof}

\begin{lemma}\label{lem:degree2}
If $G$ is a digraph with maximum in-degree at most two and with no two independent cycles, then $\phi_m(G)=\nu+1$.
\end{lemma}

\begin{proof}
This results from Lemma~\ref{lem:lowerbound}, Lemma~\ref{lem:andortotalorder} and the following basic observation: if $G$ is of maximum in-degree at most two, then every monotone Boolean network with $G$ as interaction graph is an and-or-network. 
\end{proof}

Let us now exhibit, for every $k$, a strongly connected digraph $G$ verifying the conditions of the previous lemma and such that $\nu=k$. For every $n\geq 1$, let $\calK_n$ be the directed graph defined as follows. The set of vertices is $N\times N$, with $N=\{0,1,\dots,n-1\}$, and the set of arcs is defined as follows, where sums are computed modulo $n$:
\begin{enumerate}
\item
for all $i,j\in N$, there is an arc from $(i,j)$ to $(i,j+1)$. 
\item
for all $i,j\in N$ with $i\neq j$, there is an arc from $(i,i)$ to $(j,i)$.  
\end{enumerate}

\begin{figure}
\begin{center}
\newcounter{aux}
\newlength{\radio}

\def\n{4}
\begin{tikzpicture}
\setlength{\radio}{3.5cm}
\foreach \x [evaluate =\x as \xx using {int(\x-1)}] in {1,...,\n}{
	\foreach \y [evaluate = \y as \yy using {int(mod(\x+\y-2,\n))}] in {1,...,\n}{
		\ifnum \xx=\yy
			\node[main node,ultra thick] (\xx\yy) at ({(1+2*\n*\x+2*\y)*180/(\n*\n)-90}:\radio) {\xx,\yy};
		\else
			\node[main node] (\xx\yy) at ({(1+2*\n*\x+2*\y)*180/(\n*\n)-90}:\radio) {\xx,\yy};
		\fi
	}
}
\foreach\x [evaluate =\x as \xx using {int(\x-1)}] in {1,...,\n}{
	\foreach \y  [evaluate = \y as \yy using {int(mod(\x+\y-2,\n))}] [evaluate=\y as \nyy using {int(mod(\y+\x-1,\n))}] in {1,...,\n}{
		\ifnum \xx=\yy
		\else  \path[arco](\xx\xx) edge (\yy\xx);\fi
		\ifnum \nyy=\xx
			\path[arco,ultra thick,bend left=65](\xx\yy) edge (\xx\nyy);
		\else
			\path[arco,ultra thick,bend right=10](\xx\yy) edge (\xx\nyy);
		\fi
	}
}
\end{tikzpicture}
{\caption{$\calK_4$.}}
\end{center}
\end{figure}

\begin{proposition}\label{pro:Kn}
For all $n\geq 1$, we have $\nu(\calK_n)=\tau(\calK_n)=n$ and $\nu^*(\calK_n)=1$. Furthermore, $\calK_n$ is of maximum in-degree at most two and no two independent cycles, so that $\phi_m(\calK_n)=n+1$. 
\end{proposition}

\begin{proof}
Foremost, each vertex $(i,j)$ of $\calK_n$ with $i\neq j$ has exactly two in-neighbors, namely $(i,j-1)$ and $(j,j)$ (sums and differences are computed modulo $n$). Also, each vertex $(i,i)$ has a unique in-neighbor, namely $(i,i-1)$. Thus $\calK_n$ has maximal in-degree at most two. 

\medskip
For each $i\in N$, the set of vertices $\{i\}\times N$ induces a cycle of length $n$, denoted $C_i$. These cycles $C_i$ are obviously pairwise vertex-disjoint, thus $\nu(\calK_n)\geq n$. Furthermore, since each vertex of $C_i$ distinct from $(i,i)$ has a unique out-neighbor, we deduce that every cycle $C$ meeting $C_i$ contains the vertex $(i,i)$. We deduce that the set of vertices $I=\{(i,i)\,|\,i\in N\}$ is a feedback vertex set. Thus $\tau(\calK_n)\leq n$, and we deduce that $\nu(\calK_n)=n=\tau(\calK_n)$.  

\medskip
We will prove that $\calK_n$ has no two independent cycles. Let $F_1$ and $F_2$ be two vertex-disjoint cycles of $\calK_n$. Suppose, for a contradiction, that $F_1$ and $F_2$ are independent. Let $(i_1,i_1)\in F_1\cap I$ and $(i_2,i_2)\in F_2\cap I$ be such that $|i_1-i_2|$ is minimum. Without loss of generality, we can suppose that $i_2<i_1$. Since there is an arc from $(i_2,i_2)$ to $(i_1,i_2)$, $(i_1,i_2)$ is not a vertex of $F_1$, thus $F_1\neq C_{i_1}$ and we deduce that $F_1$ contains an arc from a vertex $(j,j)$ to $(i_1,j)$, with $j\neq i_1,i_2$. Let $P$ be the path from $(i_1,j)$ to $(i_1,i_1)$ contained in $C_{i_1}$. Clearly, this path $P$ is also contained in $F_1$. Thus, if $0\leq j<i_2$ or $i_1<j\leq n-1$, then $(i_1,i_2)$ is a vertex of $P$, a contradiction. We deduce that $i_2<j<i_1$. Thus $|j-i_2|<|i_1-i_2|$ and since $(j,j)\in F_1\cap I$, this contradicts our choice of $(i_1,i_1)$ and $(i_2,i_2)$. We deduce that $\calK_n$ has no two independent cycles. Thus, by Lemma~\ref{lem:degree2}, $\phi_m(\calK_n)=\nu+1$. 

\medskip
Let us now prove that $\nu^*=1$. Let $F_1$ and $F_2$ be two vertex-disjoint cycles. Since they are not independent, there exists an arc arc $uv$ between the two cycles, say from $F_1$ to $F_2$. This arc then constitutes a principal path from $F_1$ to $F_2$. Since $v$ has in-degree exactly $2$, every principal path ending in $v$ contains $uv$. Thus $uv$ is actually the unique principal path ending in $v$, and thus $\{F_1,F_2\}$ cannot be a special packing.
\end{proof}

\subsection{Optimal exponential lower bound in {\boldmath$\nu^*$\unboldmath}}

\begin{lemma}\label{lem:lowerbound2}
For every digraph $G$, 
\[
2^{\nu^*}\leq \phi_m(G). 
\]
More precisely, there exists a monotone Boolean network $f$ with $G$ as interaction graph such that $f$ has at least $2^{\nu^*}$ fixed points and $f_v=\mathrm{cst}=0$ for every source $v$ of~$G$.  
\end{lemma}

\begin{proof}
Suppose that the vertex set of $G$ is $V=[n]$. Let $\mathcal{P}$ be a special packing of size $\nu^*$. Let $U$ be the set of vertices that belong to some cycle of the packing, and let $S$ be the set of sources of $G$. For every $v\in U$, we denote by $v'$ the unique in-neighbor of $v$ that belongs to the packing. Let $f:\B^n\to\B^n$ be defined by
\[
\begin{array}{ll}
f_v(x)=x_{v'}\lor \bigwedge_{u\in N^-_G(v)\setminus\{v'\}}x_u&\quad\forall v\in U\\[2mm]
f_v(x)=0&\quad\forall v\in S\\[2mm]
f_v(x)=\bigwedge_{u\in N^-_G(u)}x_u&\quad\forall v\in V\setminus (U\cup S).
\end{array}
\]
Clearly, $f$ is monotone and its interaction graph is $G$. 

\medskip
Let $\mathcal{I}\subseteq \calP$ be a subset of cycles in the packing and let $I$ be the set of vertices that belongs to a cycle of $\mathcal{I}$. Let $J=\{v_1,v_2,\dots,v_r\}$ be a maximal sequence of vertices in $V\setminus I$ such that $N^-_G(v_1)\subseteq I$ and $N^-_G(v_i)\subseteq I\cup\{v_1,\dots,v_{i-1}\}$ for all $1<i\leq r$ ($J$ may be empty). By construction, every cycle of $G$ intersecting $J$ also intersects $I$, so $J\cap U=\emptyset$. Let us defined $x^{\mathcal{I}}\in\B^n$ by
\[
x^\mathcal{I}_v=
\begin{cases}
1&\text{if }v\in I\cup J\\
0&\text{otherwise.}
\end{cases}
\]

\medskip
Suppose that $\mathcal{I}'\subseteq \calP$ and $\mathcal{I}'\neq\mathcal{I}$. If $C$ is a cycle of $\mathcal{I}'\setminus \mathcal{I}$ and $v$ is a vertex of $C$ then $v\in U\setminus I$ and since  $J\cap U=\emptyset$ we have $x^{\mathcal{I}}_v=0$. Since $x^{\mathcal{I'}}_v=1$, this show that $x^\mathcal{I}\neq x^{\mathcal{I}'}$, and if $\mathcal{I}\setminus \mathcal{I}'\neq\emptyset$ we obtain the same conclusion with similar arguments. Thus the $2^{\nu^*}$ possible choices of $\mathcal{I}$ gives $2^{\nu^*}$ distinct points $x^{\mathcal{I}}$. It is thus sufficient to prove that $x^{\mathcal{I}}$ is a fixed point. 

\medskip
To simplify notations we set $x=x^\mathcal{I}$. First, for all $v\in I$ we have $x_{v'}=1$ and thus $f_v(x)=1=x_v$. Then, with a straightforward induction on $i=1,\dots,r$, we prove that $f_{v_i}(x)=1=x_{v_i}$. Consequently, $f_v(x)=1=x_v$ for all $v\in I\cup J$. Furthermore, if $v\in S$ then $f_v(x)=0$ by definition, so $v\not\in I\cup J$ and thus $f_v(x)=0=x_v$. Next, let $v\in V\setminus (U\cup J\cup S)$. If $f_v(x)=1$, then $N^-_G(v)\subseteq I\cup J$ and this contradicts the maximality of $J$, thus $f_v(x)=0=x_v$. Finally, suppose that $v\in U\setminus I$, so that $x_v=0$, and let $C$ be the cycle of $\calP$ containing $v$. If $N^-_G(v)\cap (I\cup J)=\emptyset$, then $f_v(x)=0=x_v$. So suppose that there exists $u\in N^-_G(v)\cap (I\cup J)$. Then there exists a shortest path from $I$ to $u$ with only vertices in $I\cup J$, and with the arc $uv$ this gives a principal path $P$ from a cycle of $\mathcal{I}$ to $v$. Since $\mathcal{P}$ is a special packing, there exists another principal path $P'=u_1,\dots,u_p,v$ such that either $u_1$ is a source or $u_1$ belongs to $C$. In both cases, $x_{u_1}=0$, and since the internal vertices $u_i$, $1<i\leq p$, are not in $U\cup S$, we deduce, by induction on $i$, that $f_{u_i}(x)=0=x_{u_i}$. Thus $x_{u_p}=0$ and since $u_p\neq v'$ we have $f_v(x)=0=x_v$. 
\end{proof}

The lower bound $2^{\nu^*}$ gives the other direction of the implication (\ref{eq:oneway}) and we thus get the following characterization. 

\begin{theorem}
For every digraph $G$,
\begin{equation*}
\phi_m(G)=2^\tau~\iff~\nu^*=\tau. 
\end{equation*}
\end{theorem}

The most simple example of digraph such that $\nu^*=\tau$ is the union of $k$ disjoint cycles. See Figures~\ref{fig:nustar} for families of strongly connected digraphs with this property.

\subsection{Introduction of the circumference in the lower bound}

In this section, we prove that if a digraph $G$ has many vertex-disjoint {\em short} cycles, then we can construct a monotone Boolean network with interaction graph $G$ and with many fixed points. 

\medskip
For that we use the lower bound $2^{\nu^*}$ and a result about domination in digraphs. In a digraph $G=(V,A)$, a {\em dominating set} is a set of vertices $D\subseteq V$ such that every vertex in $V\setminus D$ has an in-neighbor in $D$. Using probabilistic arguments, Lee \cite{L98} proved the following result: {\em If $G$ is of minimum in-degree at least one, then $G$ has a dominating set of size at most $3|V|/2$}. We use this result to prove the next lemma, which is actually a generalization (we recover the result of Lee by taking $U=V$, $k=|V|$, $\ell=1$, and  by assuming that $G$ has minimal in-degree at least one).

\begin{lemma}\label{lem:domination}
Let $G$ be a digraph, let $U$ be non-empty subset of vertices of $G$ and let $(U_1,\dots,U_k)$ be a partition of $U$. For every $i\in [k]$, let $U'_i$ be a non-empty subset of $U_i$ of size at most $\ell$. Then there exists a subset $I\subseteq [k]$ of size at least $k/3^\ell$ such that for every $i\in I$ and every $u\in U'_i$ one of the following condition holds:
\begin{enumerate}
\item
$G$ has a path from $U\setminus\bigcup_{j\in I} U_j$ to $u$ without internal vertex in $U$.
\item
$G$ has no path from $U\setminus U_i$ to $u$. 
\end{enumerate}
\end{lemma}

\begin{proof}
We proceed by induction on $\ell$ assuming that $\ell=\max_{i\in[k]}|U'_i|$.

\medskip
If $\ell=1$ then for every $i\in[k]$ there exists $u_i\in U_i$ such that $U'_i=\{u_i\}$. Let $H$ be the digraph on $[k]$ with an arc $ij$ if and only if $i\neq j$ and $G$ has a path from $U_i$ to $u_j$ without internal vertex in $U$. Let $S$ be the set of $i\in [k]$ such that $u_i$ is a source of $H$. If $|S|\geq k/3$ then $I=S$ has the desired properties. So suppose that $|S|< k/3$. Let $H'$ be obtained from $H$ by adding an arc $ij$ for every $j\in S$ and $i\in [k]\setminus S$. In this way, the minimal in-degree of $H'$ is at least one. Consequently, by the result of Lee \cite{L98} mentioned above, $H'$ has a dominating set $D$ of size at most $2k/3$. Thus the size of $I=[k]\setminus D$ is at least $k/3$. Furthermore, by construction: for every $i\in I\setminus S$, there exists $j\in D$ such that $G$ has a path from $U_j$ to $u_i$ without internal vertex in $U$; and for every $i\in I\cap S$, $G$ has no path from $U\setminus U_i$ to $u_i$. Thus $I$ has the desired properties. This proves the case $\ell=1$.

\medskip
Now suppose that $\ell\geq 2$ and for every $i\in [k]$, let $u_i\in U'_i$. According to the case $\ell=1$ there exists a subset $I\subseteq [k]$ of size at least $k/3$ such that: 
\begin{equation}\label{eq:I}
\begin{array}{p{12cm}}
{\em For every $i\in I$, either $G$ has a path from $U\setminus \bigcup_{j\in I} U_j$ to $u_i$ without internal vertex in $U$, or $G$ has no path from $U\setminus U_i$ to $u_i$.} 
\end{array}
\end{equation}
Now, for every $i\in I$ we set $U''_i=U'_i$ if $U'_i=\{u_i\}$ and $U''_i=U'_i-\{u_i\}$ otherwise. Then $\max_{i\in[k]}|U''_i|=\ell-1$. Hence, by induction hypothesis, there exists a subset $I'\subseteq I$ of size at least $|I|/3^{\ell-1}$ such that: 
\begin{equation}\label{eq:I'}
\begin{array}{p{12cm}}
{\em For every $i\in I'$ and every $u\in U''_i$, either $G$ has a path from $U\setminus \bigcup_{j\in I'} U_j$ to $u$ without internal vertex in $U$, or $G$ has no path from $U\setminus U_i$ to $u$.}
\end{array}
\end{equation}
So $|I'|\geq |I|/3^{\ell-1}\geq k/3^\ell$ and we deduce from (\ref{eq:I}) and (\ref{eq:I'}) that $I'$ has the desired properties. 
\end{proof}

We are now in position to prove the following lower bound, using the preceding lemma and the lower bound $2^{\nu^*}$. 

\begin{lemma}
If a digraph $G$ has $k$ disjoint cycles of length at most $\ell$ then   
\[
2^{k/3^\ell}\leq \phi_m(G). 
\]
\end{lemma}

\begin{proof}
Suppose that the vertex set of $G$ is $[n]$, and let $C_1,\dots,C_k$ be disjoint cycles of $G$ of length at most $\ell$. For all $i\in[k]$, let $U_i$ be the set of vertices of $C_i$, and let $U=U_1\cup\cdots\cup U_k$. According to Lemma \ref{lem:domination}, there exists a subset $I\subseteq [k]$ of size at least $k/3^{\ell}$ such that  for every $i\in I$ and every $u\in U_i$, either $G$ has a path from $U\setminus \bigcup_{j\in I} U_j$ to $u$ without internal vertex in $U$, or $G$ has no path from $U\setminus U_i$ to $u$. 

\medskip
Let $W=U\setminus \bigcup_{i\in I} U_i$ and let $G'$ be obtained from $G$ by removing every arc $vw$ with $w\in W$. Let us prove that $\calP=\{C_i:i\in I\}$ is a special packing of $G'$. So let $i\in I$, $u\in U_i$ and suppose that $G'$ has a principal path $P$ from a cycle $C_j\neq C_i$ to $u$. Then $G$ has a path from $U-U_i$ to $U_i$, and we deduce that $G$ has a path from $W$ to $u$ without internal vertex in $U$, and since every vertex of $W$ is a source of $G'$, we deduce that $G'$ has a principal path from a source to $u$. Thus $\calP$ is indeed a special packing of $G'$ of size $|I|$. Thus according to Lemma~\ref{lem:lowerbound2} there exists a monotone Boolean network $f'$ with interaction graph $G'$ such that $f'$ has at least $2^{|I|}$ fixed points and $f'=\mathrm{cst}=0$ for every source $v$ of $G'$. 

\medskip 
Let $f:\B^n\to\B^n$ be the Boolean network defined for all $v\in[n]$ by 
\[
f_v(x)=
\begin{cases}
f'_v(x)&\textrm{if }v\not\in W\\
\bigwedge_{u\in N^-_G(v)}x_u&\text{otherwise.}
\end{cases}
\]
Then $f$ is monotone and $G$ is the interaction graph of $f$. Furthermore, if $x$ is a fixed point of $f'$ then $f_v(x)=f'_v(x)=x_v$ for all $v\notin W$ and $x_W=0$. Since if $v\in W$ then $N^-_G(v)\cap W\neq\emptyset$ we deduce that $f_v(x)=0=x_v$. Thus $x$ is a fixed point of $f$. Thus $f$ has at least as many fixed points as $f'$ and this completes the proof. 
\end{proof}

The {\em circumference} $c=c(G)$ of a digraph $G$ is the maximum length of a chordless cycle of $G$. Since every digraph $G$ has always a packing of $\nu$ chordless cycles, we deduce from the preceding lemma that $2^{\nu/3^c}\leq \phi_m(G)$. Furthermore, the set of vertices that belong to a packing of $\nu$ chordless cycles is always a feedback vertex set. Thus we always have $\tau\leq c\nu$ and has a consequence, for every digraph $G$, 
\[
2^{\nu/3^c}\leq \phi_m(G)\leq 2^{c\nu}.
\]
This lower bound can be much more better than the two previous one. As an example, for the transitive tournament on $n$ vertices with loop on each vertex we have $\nu^*=c=1$ and $\nu=n$. 

\medskip
Summarizing the previous results we get the following. 

\begin{theorem}[Lower bounds for monotone Boolean networks]\label{thm:lowerbound}
For every digraph $G$, 
\[
\max\{\nu+1,2^{\nu^*},2^{\nu/3^c}\}\leq\phi_m(G).
\]
\end{theorem}

\subsection{The undirected case}

We say that a digraph $G$ is {\em symmetric} if, for every distinct vertices $u$ and $v$, either $uv$ and $vu$ are not arcs, or both $uv$ and $vu$ are arcs. Hence, a loop-less symmetric digraph $G$ can be regarded as a simple {\em undirected} graph, simply by regarding every $2$-cycle (cycle of length two) as an undirected edge. Since a loop-less symmetric digraph $G$ has circumference $c=2$ we have $2^{\nu/9}\leq\phi_m(G)$ and this contrasts with the general case where the tight linear lower bound $\nu+1$ holds. Below, we prove something stronger: we prove that $\nu/6\leq\nu^*$, from which we immediately get $2^{\nu/6}\leq\phi_m(G)$. This connection between $\nu$ and $\nu^*$ also contrasts with the general case, where we can have $\nu^*=1$ and $\nu$ arbitrarily large (cf. Proposition~\ref{pro:Kn}). 

\begin{proposition}
For every loop-less symmetric digraph $G$,
\[
\nu/6\leq\nu^*\leq \nu\leq\tau\leq 2\nu.  
\]
\end{proposition}

\begin{proof}
The inequality $\tau\leq2\nu$ comes from the fact that the circumference of $G$ is $c=2$, so the only thing to prove is $\nu/6\leq\nu^*$. Below, we view $G$ as a simple undirected graph, that is, we view Êa cycle of length two between two vertices $u$ and $v$ has an undirected edge $uv$. If $uv$ is an edge, we say that $u$ and $v$ are {\em adjacent}, and we say that a vertex $w$ {\em covers} $uv$ if $w=u$ or $w=v$. A {\em matching} is a set of vertex-disjoint edges. Clearly, $\nu$ is the maximum size of a matching in $G$. We say that a vertex belongs to a matching $M$ if it covers one of the edge of $M$.

\medskip
Suppose that $G$ has a matching $M$ satisfying the following property:
\[
\tag{$*$}
\begin{array}{p{12cm}}
{\em If $G$ has an edge $uv$ where $u$ and $v$ cover different edges of the matching, then $v$ is adjacent to some vertex that does not belong to the matching.}
\end{array}
\]
We claim that $M$ is then a {\em special matching} (that is, the cycles of length two corresponding to the edges of $M$ form a special packing). Indeed, suppose that $M=\{u_1v_1,\dots,u_kv_k\}$ and suppose that $G$ has a principal path $P$ from $u\in\{u_i,v_i\}$ to $v\in\{u_j,v_j\}$ for some distinct $i,j\in[k]$. If $P$ is of length at least two, then the vertex $w$ adjacent to preceding $v$ in $P$ is not in the matching, and $P'=vwv$ is a principal path from $\{u_j,v_j\}$ to $v$. If $P$ is of length one, then $P$ is reduced to the arc $uv$ and we deduce from the property ($*$) that $v$ is adjacent to a vertex $w$ that is not in the matching, and thus, as previously, $P'=vwv$ is a principal path from $\{u_j,v_j\}$ to $v$. Thus $M$ is indeed a special matching. Therefore, to prove the proposition it is sufficient to prove that $G$ has a matching of size at least $\nu/6$ with the property ($*$).  

\medskip
Let $M_1=\{u_1v_1,\dots,u_\nu v_\nu\}$ be a matching of size $\nu$, and let $G_1$ be the simple undirected graph whose vertices are the edges of $M_1$ and with an edge between $u_iv_i$ and $u_jv_j$ if $G$ has an edge $uv$ with $u\in \{u_i,v_i\}$ and $v\in \{u_j,v_j\}$. Let $I_1$ be the set of isolated vertices of $G_1$. Hence, in $G$, we have the following property:
\begin{equation}\label{eq:undirected1}
\begin{array}{c}
\text{\em For every $u_iv_i\in I_1$, $u_i$ and $v_i$ are adjacent to no vertex in $M_1\setminus \{u_iv_i\}$}. 
\end{array}
\end{equation}
Now, since $G_1\setminus I_1$  has no isolated vertex, it contains a dominating set $D_1$ of size at most $(\nu-|I_1|)/2$ (this is a basic result about domination in graphs, see for instance \cite[Theorem 2.1]{HHS98}). Thus the complement $M_2=M_1\setminus (I_1\cup D_1)$ has size at least $(\nu-|I_1|)/2$. Hence, by construction, in $G$, for every edge $u_iv_i\in M_2$, $u_i$ or $v_i$ (or both) is adjacent to some vertex that belongs to the matching $D_1$. So without loss of generality, we can suppose that $G$ has the following property: 
\begin{equation}\label{eq:undirected2}
\begin{array}{c}
\text{\em For every $u_iv_i\in M_2$, $v_i$ is adjacent to some vertex in $D_1$}. 
\end{array}
\end{equation}
Let $G_2$ be the {\em digraph} whose vertices are the edges of the matching $M_2$ and with an arc {\em directed} from $u_iv_i$ to $u_jv_j$ if $G$ has an edge $uv$ with $u\in \{u_i,v_i\}$ and $v=u_j$. Let $I_2$ be the set vertices of in-degree zero in $G_2$. Hence, in $G$, we have the following property:
\begin{equation}\label{eq:undirected3}
\begin{array}{c}
\text{\em For every $u_iv_i\in I_2$, $u_i$ is adjacent to no vertex in $M_2\setminus\{u_iv_i\}$}. 
\end{array}
\end{equation}
Let $\tilde G_2$ be obtained from $G_2$ by adding an arc from every $u_iv_i\in M_2$ to every $u_jv_j\in I_2$, with $i\neq j$. Since $\tilde G_2$ has minimal in-degree at least one, by the result of Lee \cite{L98} mentioned above, it has a dominating set $D_2$ of size at most $3|M_2|/2$. Thus the complement $M_3=M_2\setminus D_2$ has size at least $|M_2|/3$. Hence, by construction, in $G$ we have the following property: 
\begin{equation}\label{eq:undirected4}
\begin{array}{c}
\text{\em For every $u_iv_i\in M_3\setminus I_2$, $u_i$ is adjacent to some vertex in $D_2$}. 
\end{array}
\end{equation}

\medskip
We have  
\[
|M_3|\geq \frac{|M_2|}{3}\geq \frac{\nu-|I_1|}{6}\geq \frac{\nu}{6}-|I_1|, 
\]
thus it is sufficient to prove that $I_1\cup M_3$ has the property ($*$). Suppose that $G$ has an edge $uv$ where $u$ and $v$ cover different edges of $I_1\cup M_3$, say $u_iv_i$ and $u_jv_j$ respectively. Then, according to (\ref{eq:undirected1}), $u_iv_i$ and $u_jv_j$ are not in $I_1$, and thus they belong to $M_3$. Furthermore, if $v=v_j$ then since $M_3\subseteq M_2$ we deduce from (\ref{eq:undirected2}) that $v$ is adjacent to a vertex not in $I_1\cup M_3$. If $v=u_j$ then we deduce from (\ref{eq:undirected3}) that $u_jv_j\not\in I_2$. Hence, according to  (\ref{eq:undirected4}), we deduce that $v=u_j$ is adjacent to a vertex not in $I_1\cup M_3$. Thus, in every case, $v$ is adjacent to a vertex not in $I_1\cup M_3$. We prove with similar arguments that $u$ is adjacent to a vertex not in $I_1\cup M_3$. Thus $I_1\cup M_3$ has the property ($*$) and this completes the proof. 
\end{proof}

\section{Upper bound for signed digraphs}\label{sec:signedgraph}

The signed interaction graph of a Boolean network has been introduced in Section~\ref{intro}, with some associated parameters. The next paragraph recalls these concepts for convenience, and introduces some additional ones.

\medskip
A {\em signed digraph} is a couple $(G,\sigma)$ where $G$ is a digraph and $\sigma$ is an arc-labeling function taking its values in $\{-1,0,1\}$. We denote by $(G,+)$ (resp. $(G,-)$) the signed digraph $(G,\sigma)$ where $\sigma$ gives a positive (resp. negative) sign to each arc. An {\em undirected cycle} is a sequence of distinct vertices $v_0,\dots,v_\ell$ such that $G$ has an arc from $v_{k}$ to $v_{k+1}$ or from $v_{k+1}$ to $v_{k}$ for all $0\leq k\leq\ell$ (where $k+1$ is computed modulo $\ell+1$). The sign of a (directed or undirected) cycle is the product of the signs of its arcs. We say that $(G,\sigma)$ is {\em balanced} if all its directed and undirected cycles are positive. The notion of balance is central in the study of signed digraphs \cite{H59,Z82,Z12}. The {\em signed interaction graph} of a Boolean network $f:\B^n\to\B^n$ is the signed digraph $(G,\sigma)$ where $G$ is the interaction graph of $f$ and where $\sigma$ is the arc-labeling function defined for all arc $uv$ of $G$ by 
\[
\sigma(uv)=
\left\{
\begin{array}{rl}
 1 &\text{if $f_v(x)\leq f_v(x+e_u)$ for all $x\in\B^n$ with $x_u=0$},\\
-1 &\text{if $f_v(x)\geq f_v(x+e_u)$ for all $x\in\B^n$ with $x_u=0$},\\
 0 &\text{otherwise.}
\end{array}
\right.
\] 
Given a digraph $(G,\sigma)$, we denote by $\phi(G,\sigma)$ the maximum number of fixed points in a Boolean network whose signed interaction graph is (isomorphic to) $(G,\sigma)$. Furthermore, we denote by $\tau^+=\tau^+(G,\sigma)$ the minimum size of a set of vertices intersecting every {\em non-negative} cycle, and we denote by $\nu^+$ the maximum number of vertex-disjoint {\em non-negative} cycles of $(G,\sigma)$. We have, obviously, $\nu^+\leq \tau^+\leq\tau$ and $\nu^+\leq\nu\leq \tau$. One of the authors \cite{A08} proved the following generalization of the feedback bound $2^\tau$, that we call {\em positive feedback bound} in the following: for every signed digraph $(G,\sigma)$, 
\[
\phi(G,\sigma)\leq 2^{\tau^+}.
\] 

\medskip
In this section, we discuss previous results on monotone networks in the context of signed digraphs, and we establish an upper bound, competitive with the positive feedback bound $2^{\tau^+}$, which is also based on the cycle structure. Let us start with a basic observation. Given a Boolean network $f$ with signed interaction graph $(G,\sigma)$, and a vertex $v$ in this graph, it is easy to see that $f_v$ is monotone if and only if $\sigma(uv)=1$ for every in-neighbor $u$ of $v$. As a consequence, $f$ is monotone if and only if $\sigma=\mathrm{cst}=1$, and we deduce that 
\begin{equation}\label{eq:1}
\phi(G,+)=\phi_m(G). 
\end{equation}

\medskip
To go further we need an usual operation on signed digraphs, called switch $\cite{Z82}$. Given a signed digraph $(G,\sigma)$ and a subset $I$ of vertices in $G$, the {\em $I$-switch} of $(G,\sigma)$ is the signed digraph $(G,\sigma^I)$ where $\sigma^I$ is defined for every arc $uv$ of $G$ by 
\[
\sigma^I(uv)=
\left\{
\begin{array}{rl}
\sigma(uv)&\text{if }u,v\in I\text{ or }u,v\not\in I,\\
-\sigma(uv)&\text{otherwise}. 
\end{array}
\right.
\]
Observe that every cycle $C$ of $G$ has the same sign in $(G,\sigma)$ and $(G,\sigma^I)$. As a consequence $\nu^+$ and $\tau^+$ are invariant under the switch operation. Below, we prove that $\phi$ is also invariant under the switch operation. 

\begin{lemma}\label{lem:switch}
Let $(G,\sigma)$ be a signed digraph and let $I$ be a subset of its vertices. Then
\[
\phi(G,\sigma)=\phi(G,\sigma^I). 
\]
\end{lemma}

\begin{proof}
Let $f:\B^n\to\B^n$ be a Boolean network with $(G,\sigma)$ as signed interaction graph and with $\phi(G,\sigma)$ fixed points. Let $f^I:\B^n\to\B^n$ be defined by
\[
f^I(x)=f(x+e_I)+e_I. 
\]
It is easy to check that $f(x)=x$ if and only if $f^I(x+e_I)=x+e_I$. Thus $f$ and $f^I$ have the same number of fixed points. It is also easy to check that $(G,\sigma^I)$ is the signed interaction graph of $f^I$. Hence, $\phi(G,\sigma)\leq \phi(G,\sigma^I)$ so that $\phi(G,\sigma^I)\leq \phi(G,(\sigma^I)^I)$. Since $(\sigma^I)^I=\sigma$ we have $\phi(G,\sigma^I)\leq \phi(G,\sigma)$ and this completes the proof. 
\end{proof}

Given a signed digraph $(G,\sigma)$, the {\em frustration index} (sometime called {\em index of imbalance}), denoted $\lambda=\lambda(G,\sigma)$, is the minimum number of non-positive arcs in a switch of $(G,\sigma)$. A basic result in signed graph theory is that $\lambda$ can be equivalently defined as the minimum number of arcs whose deletion in $(G,\sigma)$ yield a balanced signed digraph \cite{H59}. Hence, the following assertions are equivalent: $\lambda=0$; $(G,\sigma)$ is balanced; there exists a subset of vertices $I$ such that $(G,\sigma^I)=(G,+)$. We then deduce from (\ref{eq:1}) and Lemma~\ref{lem:switch} that, 
\[
\text{$(G,\sigma)$ is balanced}
~\Rightarrow~
\phi(G,\sigma)=\phi_m(G).
\]
Hence, all the results on the maximum number of fixed points obtained in the monotone case are more generally valid for balanced digraphs.

\medskip
We now introduce an upper bound on $\phi(G,\sigma)$ that is competitive with $2^{\tau^+}$. A {\em monotone feedback vertex set} in signed digraph $(G,\sigma)$ is a feedback vertex set $I$ of $G$ such that $\sigma(uv)=1$ for every arc $uv$ of $G$ with $v\not\in I$. We denote by $\tau_m=\tau_m(G,\sigma)$ the minimum size of a monotone feedback vertex set of $(G,\sigma)$. We have, obviously, $\tau\leq\tau_m$. As an immediate application of two lemmas introduced to bound $\phi_m(G)$ we have the following property. 

\begin{lemma}\label{lem:taum}
Let $f$ be a Boolean network with signed interaction graph $(G,\sigma)$. The number of fixed points in $f$ is at most the sum of the $\nu^++1$ largest binomial coefficients $\tau_m\choose k$.
\end{lemma}

\begin{proof}
Let $I$ be a monotone feedback vertex set of $(G,\sigma)$ of size $\tau_m$. If $v\not\in I$ then $\sigma(uv)=1$ for all $u\in N^-_G(v)$, so that $f_v$ is monotone. Hence, by Lemma~\ref{lem:iso}, $\fixe(f)$ is isomorphic to a subset $X\subseteq \B^I$, and by Lemma~\ref{lem:totalorder}, $X$ has no chain of size $\nu^++2$. Hence, by the Erd\H{o}s' theorem already used in the proof of Theorem~\ref{thm:erdos}, we deduce that $|X|$ is at most the sum of the $\nu^++1$ largest binomial coefficients ${|I|\choose k}$.
\end{proof}

Contrary to $\nu^+$, $\tau^+$, $\lambda$ and $\phi$, the parameter $\tau_m$ is {\em not} invariant under the switch operation. Consider for instance the loop-less symmetric digraph $K_{n,n}$ corresponding to the complete bipartite graph with both parts of size $n$. We have $\tau_m(K_{n,n},-)=2n$ and  $\tau_m(K_{n,n},+)=\tau(K_{n,n})=n$. Furthermore, $(K_{n,n},+)$ is the switch of $(K_{n,n},-)$ by one of the two parts. Thus, given a signed digraph $(G,\sigma)$ with vertex set $V$, it makes sense consider the parameter $\tau^*_m$ defined as the minimum size of a monotone feedback vertex set in a switch of $(G,\sigma)$, that is, 
\[
\tau^*_m=\tau^*_m(G,\sigma)=\min_{I\subseteq V}\tau_m(G,\sigma^I).
\]
Graph parameters based on feedback vertex sets are thus ordered as follows:
\[
\tau^+\leq\tau\leq\tau^*_m\leq\tau_m,
\]
and $\tau^+=\tau_m$ for balanced signed digraphs.

\medskip
We have now the following upper bound, that improves the one of Lemma~\ref{lem:taum}. 

\begin{theorem}
Let $f$ be Boolean network with signed interaction graph $(G,\sigma)$. The number of fixed points in $f$ is at most the sum of the $\nu^++1$ largest binomial coefficients $\tau^*_m\choose k$.
\end{theorem} 

\begin{proof}
Let $I$ be such that $\tau^*_m=\tau^*_m(G,\sigma)=\tau_m(G,\sigma^I)$. According to Lemma~\ref{lem:taum}, $\phi(G,\sigma^I)$ is at most the sum of the $\nu^+(G,\sigma^I)+1$ largest binomial coefficients $\tau^*_m\choose k$. Since $\nu^+(G,\sigma^I)=\nu^+(G,\sigma)$ and since, by Lemma~\ref{lem:switch}, $|\fixe(f)|\leq\phi(G,\sigma)=\phi(G,\sigma^I)$, this proves the theorem. 
\end{proof}

An equivalent statement is the one given in the introduction:
\begin{equation}\label{eq:signedcase}
\phi(G,\sigma)\leq \sum_{k=\lfloor\frac{\tau^*_m-\nu^+}{2}\rfloor}^{\lfloor\frac{\tau^*_m+\nu^+}{2}\rfloor}{\tau^*_m \choose k}.
\end{equation}
This clearly improves the bound $2^{\tau^+}$ when $\tau^+=\tau^*_m$ and when the gap between $\nu^+$ and $\tau^*_m$ is large. Consider for instance the loop-less symmetric digraph $K_n$ corresponding to the complete graph on $n\geq 2$ vertices. We have $\tau^+(K_n,-)=\tau^*_m(K_n,-)=n-1<\tau_m(K_n,-)=n$ and $\nu^+(K_n,-)=\left\lfloor\frac{n}{2}\right\rfloor$. Thus, by (\ref{eq:signedcase}), $\phi(K_n,-)$ is at most the $\lfloor\frac{n}{2}\rfloor+1$ largest binomial coefficient ${n-1\choose k}$. This is much better than the positive feedback bound $2^{\tau+}$ which is, in this case, $2^{n-1}$. However, for $(K_n,-)$, the upper bound (\ref{eq:signedcase}) is still far from the exact value $\phi(K_n,-)={n\choose \lfloor\frac{n}{2}\rfloor}$ established in $\cite{GRR15}$ with a dedicated method. Actually, for $(K_n,-)$, the upper bound (\ref{eq:signedcase}) is tight only for $n=2,3$ (while the bound $2^{\tau^+}$ is tight only for $n=2$).  

\medskip
Let us finally mention that the parameter $\tau^*_m$ is connected to the transversal number $\tau$  and  frustration index $\lambda$ in the following way. Let $(G,\sigma)$ be a signed digraph and let $I$ be such that $(G,\sigma^I)$ has $\lambda$ non-positive arcs. Let $J$ be the set of vertices $v$ such that $G$ has an arc $uv$ with $\sigma^I(uv)\neq 1$. If $K$ is a feedback vertex set of $G$ then $J\cup K$ is a monotone feedback vertex set of $(G,\sigma^I)$ of size $|J\cup I|\leq|J|+|I|\leq \lambda +\tau$. Thus, for every signed digraph $(G,\sigma)$, 
\[
\tau\leq \tau^*_m\leq\tau+\lambda.
\]   

\section{Concluding remarks}\label{sec:conclusion}

Let us first remark that the parameters $\nu^*$, $\nu$ and $\tau$ are invariant under the subdivision of arcs, and that  the parameters $\nu^+$, $\tau^+$ and $\tau_m$ and $\tau^*_m$ are invariant under the subdivision of arcs preserving signs (that is, when signed arcs are replaced by paths with the same signs). This is often relevant in practice. For instance, in biology, a biological pathways  involving dozens of biological entities is often summarized by a single arc. Hence, in this kind of context, parameters invariant under subdivisions are particularly relevant, in contrast with the number of vertices, the circumference or the girth.  

\medskip
Let us now present a few possible future works. To make progresses on our understanding of $\phi_m(G)$, it could be interesting to determine the maximal size of a subset of $L\subseteq \B^\tau$ satisfying the four constraints of Theorem~\ref{thm:upperbound} (by Erd\H{o}s' theorem, we known the maximal size of $L$ subjects to the first two only). To initiate this study, it could be interesting to first focus on special $k$-patterns, considering the following question: {\em Given positive integers $k$ and $n$, what is the maximum size of a subset of $\B^n$ without special $k$-pattern.}

\medskip
It could be interesting to go further in the study of $\phi(G,\sigma)$, for instance by establishing lower bounds in the spirit of Theorem~\ref{thm:lowerbound}. We may also think about an improvement of  (\ref{eq:signedcase}) that generalizes (\ref{eq:upperbound}) to the signed case.

\medskip
Besides, we known that directed cycles have the Erd\H{o}s-P\'osa property: there exists a (smallest) function $h:\mathbb{N}\to\mathbb{N}$ such that $\tau\leq h(\nu)$ for every digraph \cite{RRST95}. This leads us to ask if non-negative cycles have also this property, that is: {\em Is there exists a function $h^+:\mathbb{N}\to\mathbb{N}$ such that $\tau^+\leq h^+(\nu^+)$ for every signed digraph?} It is easy to show that this is equivalent to ask if {\em even} directed cycles have the Erd\H{o}s-P\'osa property (this is true in the undirected case~\cite{T88}). 
 
\medskip
Let $\phi_m(k)$ be the maximal number of $\phi_m(G)$ for a digraph $G$ with $\nu(G)\leq k$. We known that $\phi_m(k)$ is finite for every $k$ since, for every digraph $G$, we have $\phi_m(G)\leq 2^\tau\leq 2^{h(\nu)}$, so that $\phi_m(k)\leq 2^{h(k)}$. However, the established upper-bound on $h(k)$ is very large (power tower type) and the proof is quite involved. It could be thus interesting to try to prove the finiteness of $\phi_m(k)$ directly, without use the fact that directed cycles have the Erd\H{o}s-P\'osa property. We could hope to find, in this way, the right order of magnitude of $\phi_m(k)$. Let us mention that, according to (\ref{eq:nu1}), $\phi_m(1)=2$, and according to the proposition below, $\phi_m(2)=4$. In contrast, the unique exact value of $h$ is $h(1)=3$ and is hard to obtain \cite{M93}. This shows that bounding $\phi_m$ by $\nu$ could be much more simple than bounding $\tau$ by $\nu$.

\begin{proposition}
For every digraph $G$,
\[
\nu\leq 2~\Rightarrow~\phi_m(G)\leq 4.
\]
\end{proposition}   

\begin{proof}
Let $f$ be a monotone network with interaction graph $G$. We first prove the following: {\em If $f$ has three pairwise incomparable fixed points, then $G$ has three vertex-disjoint cycles.} Let $x^1,x^2,x^3$ be three pairwise incomparable fixed points of $f$. Let 
\[
V_1=\{v:x^1_v>x^2_v\},\qquad 
V_2=\{v:x^2_v>x^3_v\},\qquad 
V_3=\{v:x^3_v>x^1_v\}.
\]
It is easy to see that $V_1$, $V_2$ and $V_3$ are pairwise disjoint and that, for $i=1,2,3$, the minimum in-degree of $G[V_i]$ is at least one. Thus $G$ has three vertex-disjoint cycles. Therefore, if $\nu\leq 2$, then $\fixe(f)$ has no antichain of size $3$ and no chain of size $4$ (cf. Lemma~\ref{lem:totalorder}). Since $\fixe(f)$ is a lattice, we deduce that $f$ has at most four fixed points (and this bound is sharp, as showed by the identity on $\B^2$). 
\end{proof}

\paragraph{Acknowledgements} We would like to thank an anonymous referee for his valuable comments. 

\bibliographystyle{plain}
\bibliography{BIB}

\end{document}